\documentclass[10pt,conference]{IEEEtran}

\makeatother
\usepackage{graphicx}
\usepackage{tikz}
\usepackage{color}
\usepackage{amsmath, amsfonts, amssymb, mathrsfs}
\usepackage[amsmath,thmmarks]{ntheorem}
\usepackage{subfigure}
\usepackage{epstopdf}

\usepackage{pgfplots}

\usepackage{arydshln}
\usepackage{mathabx}
\usepackage{cite}
\usepackage[super]{nth}
\usepackage{cite}

\usepackage{algorithm}
\usepackage{algorithmic}
\usepackage{bbm}
\usepackage{bm}

\usetikzlibrary{calc}
\usetikzlibrary{arrows}
\usetikzlibrary{patterns}
\usetikzlibrary{shapes}


\theoremstyle{plain}
\theoremheaderfont{\itshape\bfseries}
\theorembodyfont{\normalfont\itshape}
\theoremseparator{:}
\theoremsymbol{}%\vspace{0.2cm}}
\newtheorem{theorem}{Theorem}
\newtheorem{lemma}[theorem]{Lemma}
\newtheorem{corollary}[theorem]{Corollary}
\newtheorem{proposition}[theorem]{Proposition}

\newtheorem{definition}{Definition}

\theoremstyle{nonumberplain}

\theoremstyle{plain}
\theoremheaderfont{\itshape\bfseries}
\theorembodyfont{\normalfont}
\theoremseparator{:}
\theoremsymbol{}
\newtheorem{remark}{Remark}

\theoremstyle{plain}
\theoremheaderfont{\itshape\bfseries}
\theorembodyfont{\normalfont}
\theoremseparator{:}
\newtheorem{example}{Example}
\newtheorem{assum}{Assumption}

\theoremstyle{nonumberplain}
\theoremsymbol{\rule{1ex}{1ex}}

\newtheorem{proof}{Proof}
\newlength\fheight
\newlength\fwidth
\newcommand{\abs}[1]{\left| #1 \right|}
\newcommand{\norm}[1]{\| #1 \|}
\newcommand{\inn}[2]{\langle #1,#2 \rangle}
\newcommand{\lrbrace}[1]{\left\{ #1 \right\}}

\newcommand{\normi}[1]{{\left\vert\kern-0.25ex\left\vert\kern-0.25ex\left\vert #1 
		\right\vert\kern-0.25ex\right\vert\kern-0.25ex\right\vert}}
\newcommand{\inni}[2]{{\langle\kern-0.25ex\langle #1,#2
		\rangle\kern-0.25ex\rangle}}

\def\X{\mathcal{X}}

\def\d{ \mathrm{d} }								% integral d
\def\T{ \mathrm{T} }								% T - transpose

\def\nat{ \mathbb{N} }								% set of natural numbers

\def\real{ \mathbb{R} }								% real numbers

\def\Erw{ \mathbb{E} }

\def\Vari{ \mathbb{V} }
\def\Prob{ \mathbb{P} }
\def\D{ \mathcal{D} }								% Banach space  

\def\V{\mathcal{V}}

\def\X{\mathcal{X}}

\def\bx{ \mathbf{x} }									% bold x
\def\by{ \mathbf{y} }									% bold y
	
\DeclareMathOperator{\Gap}{\textnormal{Gap}}

\DeclareMathOperator{\Reg}{\textnormal{Reg}}

\DeclareMathOperator{\hCFit}{\textnormal{$h$-CFit}}
\newcommand\restr[2]{{% we make the whole thing an ordinary symbol
  \left.\kern-\nulldelimiterspace % automatically resize the bar with \right
  #1 % the function
  \right|_{#2} % this is the delimiter
  }}

\DeclareMathOperator*{\argmax}{arg\,max}
\DeclareMathOperator*{\argmin}{arg\,min}

\def\ent{\text{ent}}

\def\Z{\mathcal{Z}}

\title{Robust Online Learning for Resource Allocation - Beyond Euclidean Projection and Dynamic Fit}
\author{\IEEEauthorblockN{Ezra Tampubolon and Holger Boche}
\IEEEauthorblockA{
  Lehrstuhl f{\"u}r Theoretische Informationstechnik\\
  Technische Universit{\"a}t M{\"u}nchen, 80290 M{\"u}nchen, Germany\\
  \{ezra.tampubolon,boche\}@tum.de}
 }
\begin{document}
%\ninept
%
\maketitle
%
%\circ
\begin{abstract}
Online-learning literature has focused on designing algorithms that ensure sub-linear growth of the cumulative long-term constraint violations. The drawback of this guarantee is that strictly feasible actions may cancel out constraint violations on other time slots. For this reason, we introduce a new performance measure called $\hCFit$, whose particular instance is the cumulative positive part of the constraint violations. We propose a class of non-causal algorithms for online-decision making, which guarantees, in slowly changing environments, sub-linear growth of this quantity despite noisy first-order feedback. Furthermore, we demonstrate by numerical experiments the performance gain of our method relative to state of art. 
\end{abstract}

% ================================================================================================================
% ========== INTRODUCTION ========================================================================================
% ================================================================================================================
\section{Introduction}
\label{sec:intro}
The online learning \cite{Zinkevich2003} is an emerging paradigm aiming to solve the problem of sequential decision making in an unknown and possibly adversarial environment: Consider a time horizon $T\in\nat$. At each time slot $t\in [T]$, the decision-maker chooses an action $X_{t}$ from a compact convex set, and simultaneously the environment reveals a convex cost function $f_{t}$. This procedure results in the decision-maker suffering the loss $f_{t}(X_{t})$ resulting from applying the action $X_{t}$.

 The online learning method is ideally suited for application where the underlying problem is subject to unpredictable dynamic, such as dispatch of renewable energy having intermittent and unpredictable nature, or network applications where the task to accomplish is subject to unpredictable human participation, or applications requiring flexibility in handling heterogenity and scalability. Furthermore, applications requiring real-time decision leverage from an online learning method since the given algorithm is usually lightweight. For those reasons, online learning has become in the recent years a popular method to solve several resource allocation and management problems in several engineering fields such as economic dispatch in power systems \cite{Narayanaswamy2012,Moeini2014}, data center scheduling \cite{Lin2011,Lin2012,Chen2017}, electric vehicle charging \cite{Gan2013,Kim2014}, video streaming \cite{Joseph2012}, thermal control \cite{Zanini2010}, and fog computing in IoT \cite{Chen2018,Chen2019,Chen20192}. 
 
Classical OL deal with problems with time-invariants constraints that has to be strictly satisfied. Therefore, a projection operator is typically applied to the update. However, in practical applications (see e.g. \cite{Chen2017,Chen2018,Chen2019,Chen20192}) one usually encounter additional time-variant constraints. Moreover, the need for decentralization of the learner action in applications can not be satisfied by simply using a centralized projection operator. For those reasons, several works \cite{Mahdavi1,Yuan2018,Yu2016,Chen2017,Chen2018} propose projected primal-dual methods which ensure sub-linear growth of the regret, i.e., the cumulative distance of the generated action to the optimal ones, and the long-term fulfillment of the constraints in the sense that the sum of the constraint violations grow sub-linearly. A problem relates to this long-term guarantee is that it holds, despite substantial instantaneous constraint violations, as long as the methods generate strictly feasible actions canceling the latter.

\paragraph*{Our Contribution}
In this work, we introduce a new long-term constraint preservation performance measure called $\hCFit$. The feature of this performance measure is that $\hCFit$ avoids cancellation effects between the summands, which might occur in the simple cumulative constraint violation measure. We design a non-causal saddle-point method based on mirror descent aiming to ensure dynamic regret optimization and sub-linear growth of $\hCFit$. In particular, we can guarantee dynamic regret bound of order $\mathcal{O}(\mathbb{V}_{T}T^{1/2})$, where $\mathbb{V}_{T}$ measures the variation of the optimizers of the underlying time-varying problem, and $\hCFit$-bound of order $\mathcal{O}(T^{3/4})$. We show by numerical experiments the performance gain of our method relative to state of the art and the advantage of using a mirror map other than Euclidean projection.
\paragraph*{Relation to Prior Works}  
\begin{table*}%[ht]
	\caption{An Overview of Related Works on Online Convex Optimization}% title of Table
	\centering % used for centering table
	\begin{tabular}{c c c c c c}% centered columns (4 columns)
		\hline\hline                        %inserts double horizontal lines
		\\ [0.2ex]
		References & Long-Term Constraint Type & Feedback Noise & Regret Bound & Regret Benchmark Type & Constraint Violation Bound \\ [0.5ex]% inserts table %heading
		\hline \hline                 % inserts single horizontal line
		\cite{Zinkevich2003} & No & No & $\mathcal{O}(T^{1/2})$  & Static and Dynamic & - \\% inserting body of the table2 
		\cite{Mahdavi1,Jenatton2016,Yu2016} & $\sum_{t=1}^{T}g(x_{t})$ & No & $\mathcal{O}(T^{1/2})$ & Static & $\mathcal{O}(T^{3/4})$ resp. $\mathcal{O}(T^{1/2})$ \\
		\cite{Yu2017} & $\sum_{t=1}^{T}g_{t}(x_{t})$    & No  & $\mathcal{O}(T^{1/2})$ & Static  & $\mathcal{O}(T^{1/2})$\\
		\cite{Yuan2018} & $\sum_{t=1}^{T}[g(x_{t})]^{2}_{+}$ & No & $\mathcal{O}(T^{1/2})$  & Static  & $\mathcal{O}(T^{1/2})$\\
		This paper & $\sum_{t=1}^{T}h(g_{t}(X_{t}))$, & Martingale& $\mathcal{O}((1+\sigma^{2}+\mathbb{V}_{T})T^{1/2})$ & Dynamic & $\mathcal{O}(T^{3/4})$  \\ [1ex]      % [1ex] adds vertical space
		\hline\hline%inserts single line
	\end{tabular}

\rule{0pt}{1ex}    

	\begin{tabular}{c c}% centered columns (4 columns)
	\hline\hline                        %inserts double horizontal lines
	References & Comments \\ [1ex]% inserts table %heading
	\hline \hline                 % inserts single horizontal line
	\cite{Zinkevich2003} & Mirror Map \\% inserting body of the table2 
	\cite{Mahdavi1,Jenatton2016} & Mirror Map  \\
	\cite{Yu2016,Yu2017} & Requires Slater condition, Causal dual update, Euclidean projection \\
	\cite{Yuan2018} &  \\
	This paper & Mirror Map \\ [1ex]      % [1ex] adds vertical space
	\hline\hline%inserts single line
\end{tabular}
\label{table:nonlin}% is used to refer this table in the text
\end{table*}

	 In the absence of long-term constraints, \cite{Zinkevich2003} showed that the standard method of online mirror descent achieves $\mathcal{O}(\sqrt{T})$ regret bound (see also \cite{Shalev-Shwartz2012}), which is known to be optimal \cite{Abernethy2008}. However, their notion of regret, i.e., static regret, corresponds to the difference of the losses between the online solution and the overall best static solution in hindsight, which is weaker than ours. 
	 
	 The first work tackling the online problem with long-term constraints is \cite{Mahdavi1}. This work considers time-invariant constraint function and proposed an algorithm having $\mathcal{O}(\sqrt{T})$ regret bound, and $\mathcal{O}(T^{3/4})$ cumulative constraint violation bound. As investigated by \cite{Jenatton2016}, one can efficiently trade-off between those bound by allowing the step-size to be variable. By utilizing Slater's condition, and allowing the dual update to depend causally on the primal update, \cite{Yu2016} provides an improved $\mathcal{O}(\sqrt{T})$ bound for the cumulative constraint violation. \cite{Yuan2018} was able to provide $\mathcal{O}(\sqrt{T})$ bound for tighter long-term constraint preservation measure $\sum_{t=1}^{T}[g(X_{t})]_{+}^{2}$. However, this remarkable guarantee is achieved by allowing the dual update to utilize causal information about the primal variable.  
	 
	 To the best of our knowledge, the first work considering the online problem with time-varying constraints is \cite{Yu2017}. Based on \cite{Yu2016}, they provide a (causal) primal-dual algorithm ensuring that the static regret is of order $\mathcal{O}(\sqrt{T})$, and the cumulative constraint violation of order $\mathcal{O}(\sqrt{T})$. 
	
Until now, we only discuss works delivering static regret guarantees. The work \cite{Chen2017} proposed a projected gradient descent based algorithm for the online problem with time-varying constraints aiming to optimize the regret against the dynamic comparator. Their result relies on the assumption that two consecutive constraint functions are bounded by the slack achieved by a fixed primal action uniformly over all constraint functions. Surely both, the existence of the slack and the action, and the boundedness of the difference consecutive constraint functions are difficult to guarantee. Despite this fact, their dynamic regret bound is worse than ours ($\mathcal{O}(\mathbb{V}_{T}T^{1/2})$) since it is lower bounded by $\mathcal{O}(\mathbb{V}_{T}T^{1/3})$. The cumulative constraint violations guarantee given \cite{Chen2017} is of order $\mathcal{O}(T^{2/3})$, which is better than ours ($\mathcal{O}(T^{3/4})$). However, our performance measure to this respect, i.e., $\hCFit$, is stronger than that given in \cite{Chen2017}. A clear plus-point of \cite{Chen2017} is the application of the proposed online algorithm to proactive network allocation. \cite{Chen2018} proposed a novel adaptive algorithm for the online problem with time-varying constraints with interesting applications to the problem of computational offloading in IoT. The corresponding method possesses higher computational complexity than ours since it requires the covariance of the gradient, root, and inverse operation of a matrix. Despite of this fact, their dynamic regret guarantee ($\mathcal{O}(T^{7/8}\mathbb{V}_{T}$) and long-term constraint preservation guarantee $(\mathcal{O}(\max\lrbrace{T^{15/16},T^{7/8}\sqrt{\mathbb{V}_{T}}}))$ is worse than ours.

%Other line of works!
% e.g. Dynamic comparator ....
% K Benchmark
\paragraph*{Basic Notions and Notations}
For a real vector $a$, $[a]_{+}$ denotes the vector whose entries are the non-negative part of the entries of $a$.
%, i.e. $[a]_{+}=\max\lrbrace{0,a}$.
%For a convex subset $A\subseteq\real^{D}$, $\text{relint}(A)$ denotes the relative interior of $A$.
%, and $\norm{A}$ denotes the supremum of $\norm{x-y}$ over all $x,y\in A$ called the diameter of $A$ w.r.t. a norm $\norm{\cdot}$ on $\real^{D}$. 
The canonical projection onto a closed convex subset $A$ of an Euclidean space $\real^{D}$ is denoted by $\Pi_{A}$, i.e.: 
\begin{equation*}
\Pi_{A}(\by):=\argmin_{\bx\in A}\norm{\by-\bx}_{2}.
\end{equation*}
%The dual norm of a norm $\norm{\cdot}$ on $\real^{D}$ is denoted by $\norm{\cdot}_{*}$. 
%$F:\real^{D}\rightarrow\real^{D}$ is said to be Lipschitz continuous on an a non-empty subset $\Z\subset (\real^{D},\norm{\cdot})$ with constant $L>0$ if $\norm{F(x)-F(z)}_{*}\leq L\norm{x-z}$, $\forall x,z\in\mathcal{Z}$. 
%$F$ is said to be monotone on $\Z$ if $\inn{x_{1}-x_{2}}{F(x_{1})-F(x_{2})}\leq 0$, for all $x_{1},x_{2}\in \mathcal{Z}$. If in the latter strict inequality hold for $x_{1}\neq x_{2}$, then $F$ is said to be strictly monotone. 
%Let be $c>0$. $F$ is said to be $c$-strongly monotone on $\Z$ if $\inn{x_{1}-x_{2}}{F(x_{1})-F(x_{2})}\leq -c \norm{x_{1}-x_{2}}^{2}$, for all $x_{1},x_{2}\in \mathcal{Z}$. 
For a subspace $A$ of an Euclidean space $\real^{D}$, we denote the diameter of $A$ by:
\begin{equation*}
D_{A}:=\sup_{\bx,\by\in A}\norm{\bx-\by}_{2}
\end{equation*}
Let $(\X,\norm{\cdot})$ be a normed space and $A,B\subseteq \X$. We denote $A-B:=\lrbrace{x-y:~x\in A,~y\in B}$, and $\norm{A}:=\sup_{x\in A}\norm{x}$. In this work we assume that a probability space $(\Omega,\Sigma,\Prob)$ and a filtration $\mathbb{F}:=(\mathcal{F}_{n})_{n\in\nat_{0}}$ therein are given. 
%For ease of notations, we denote for each $n\in\nat_{0}$ the conditional expectation $\Erw[\cdot|\mathcal{F}_{n}]$ given $\mathcal{F}_{n}$ simply by $\Erw_{n}[\cdot]$. 
% ================================================================================================================
% ========== Main Result ==============================================================================
% ================================================================================================================
\section{Problem Formulation}
%We assume that there are $M$ learner.
We begin by stating the online learning problem in the classical setting: 
\subsection{Online Learning with Classical Aggregate constraint goal}
\label{Subsec:OnlineLearn}
At each time $t$, a learner decides for an option for action $X_{t}$ from an apriori known compact convex set $\X\subset\real^{D}$, which we refer to as a \textit{feasible set}. Subsequently, nature chooses the \textit{loss function} $f^{(i)}_{t}$ and charges the learner with loss $f_{t}(X_{t})$. As already noticed in prior works, it is advantageous from a practical point of view to take into account a time-varying \textit{penalty function} $g_{t}=(g_{t}^{r})_{r\in [R]}$ chosen by nature and revealed to the learner at a time $t$. This function leads to a time-varying \textit{constraint} $g_{t}(X_{t})\leq 0$. In an online learning setting, one often assumes additionally that the learner can extract information about $f_{t}$ and $g_{t}$ via access to the first-order oracle in order to choose the action for the time step $t+1$. Although this assumption is sometimes not realistic, investigation respective to this case is usually a stepping stone for designing methods in the bandit case, i.e., in the case where the learner has only access to the immediate objective- and constraint value (see e.g. Chapter 4 in \cite{Shalev-Shwartz2012} and \cite{Mahdavi1}).

Given a time horizon $T\in\nat$. The goal of the learner is to find a sequence $(X_{t})_{t=1}^{T}$ in the feasible set $\X$ that minimizes the loss $f_{t}(X_{t})$ and simultaneously fulfills the constraint $g_{t}(X_{t})\leq 0$. Since the learner cannot look into the future and therefore has to decide on her next action utilizing the current information about the loss and penalty function, the problem stated before is intractable. For this reason, one may consider a more realistic goal of finding a sequence that minimizes the time-average loss $\sum_{t=1}^{T}f_{t}(X_{t})/T$, and that ensures the fulfillment of the constraint on average over time $\sum_{t=1}^{T}g_{t}(X_{t})/T\leq 0$.

\subsection{Beyond Aggregate Constraint}
 One crucial issue about the latter goal concerning the constraint fulfillment $\sum_{t=1}^{T}g_{t}(X_{t})/T\leq 0$ is that it does not consider the possibility that the summands can cancel each other out: As long as $g_{t}(X_{t})$ is negative and small enough for specific time slots $t\in[T]$, large $g_{t}(X_{t})$ for another time slots $t\in[T]$ is admissible for the goal $\sum_{t=1}^{T}g_{t}(X_{t})/T\leq 0$.  
In order to resolute this issue, we propose a new online learning goal, that is:
\begin{equation}
\label{Eq:aoaoshshshsjddgdggdd}
\min_{(X_{t})_{t=1}^{T}\subset\X}\sum_{t=1}^{T}f_{t}(X_{t})\quad\text{s.t. }\sum_{t=1}^{T}h(g_{t}(X_{t}))\leq 0
\end{equation}
for a monotonically increasing function $h:\real\rightarrow\real$. In case that $h$ is also non-negative, this function ensures that cancellation between summands cannot occur since they are all non-negative and that the ordering between values of $g_{t}$ remains preserved. 

An example of $h$ is $h(\cdot)=[\cdot]_{+}$. This choice leads to the constraint $\sum_{t=1}^{T}[g_{t}(x_{t})]_{+}/T\leq 0$, that is stronger than $\sum_{t=1}^{T}g_{t}(x_{t})/T\leq 0$. Another example of $h$ is $h(\cdot)=[\cdot]^{p}_{+}$ with $p> 1$ leading to the constraint $\sum_{t=1}^{T}[g_{t}(x_{t})]^{p}_{+}/T\leq 0$. With increasing $p$, $[\cdot]^{p}_{+}$ penalizes large values of $g_{t}$ with the cost of loosening the sensitivity of the sum for small, non-negative values of $g_{t}$.  
%\subsection{Beyond Box Constraint} 
%\textcolor{red}{
%\begin{itemize}
%	\item Definiere Box constraint
%	\item Why Box constraint langweilig
%	\item Welchen Vorteil kann man ziehen wenn man andere Constraint art sehen?
%\end{itemize}
%}
\subsection{Noisy First-order Feedback}
\label{Subsec:Noisy}
As discussed in Subsection \ref{Subsec:OnlineLearn}, the online learning setting assumes that first-order information about the current loss function is available at each time slot. However, perfect first-order feedback is, in general, hard to obtain. Thus, we include in our model the possibility that the learner has only access to the noisy first-order oracle. Expressly, we assume that at each time $t$ and for a given action $X_{t}\in\mathcal{X}$, the learner can query an estimate of $\hat{v}_{t}$ of the (sub-)gradient $\nabla f_{t}(X_{t})$ satisfying $\Erw[\norm{\hat{v}_{t}}_{*}]<\infty$ and $\Erw[\hat{v}_{t}|\mathcal{F}_{t}]=\nabla f_{t}(X_{t})$, where $\mathcal{F}_{t}$ is an element of a filtration $\mathbb{F}:=(\mathcal{F}_{t})_{t\in\nat_{0}}$ on a probability space $(\Omega,\Sigma,\Prob)$. The canonical and commonly-used filtration in the literature is the filtration of the history of the considered iterates. Equivalently, we can model the stochastic (sub-)gradient by
\begin{equation}
\label{Eq:NoiseGrad}
\hat{v}_{t}=\nabla f(X_{t})+\xi_{t+1},
\end{equation}
where $(\xi_{t})_{t\in\nat}$ is a $\real^{D}$-valued \emph{$\mathbb{F}$-martingale difference sequence}, i.e. it is \emph{$\mathbb{F}$-adapted}, in the sense that  $\xi_{t}$ is $\mathcal{F}_{t}$-measureable for all $t\in\nat$, and that its members are \emph{conditionally mean zero}, in the sense that $\Erw[\xi_{t}|\mathcal{F}_{t-1}]=0$, for all $t\in\nat$.
\subsection{Applications}
In order to show the practical relevance of the the aspects discussed above (especially: noisy feedback and other notion of aggregate constraint), we give in the following some specific resource allocation examples.
%%%%%%%%%%%%%%%%%%%%%%%%%%%%%%%%%%%%%%%%%%%%%%%%%%%%%%%%%%%%%%%%%%%%%%%%%%%%%%%%%%%%%%%%%%%%%%%%%%%%%%%%%%%%%%%%%%%%%%%%%%%%%%%%%%%%%%%%%%%%%%%%%%%%%%%%%%%%%%%%%%%%%%%%%%%%%%%%%%%%%%%%%%%%%%%%%%%%%%%%%%%%%%%%%%%%%%%%%%%%%%% 
\begin{example}[Economic Dispatch%in Power Systems
	]
	\label{Ex:Eco}
	
%\textcolor{red}{[Possible Citation: Giannakis, Monitoring and Optimization for Power Grids: A Signal Processing Perspective]}	
%\textcolor{red}{[]}

Consider a system with $D$ producers (e.g. electric generator or data processing center) of a certain commodity (e.g. electrical power or data processing unit). At each time slot $t\in\nat$, the goal of economic dispatch is to decide for each $i\in [D]$ the output $X_{t}^{(i)}$ of producer $i$ causing costs $c_{t}^{(i)}(X_{t}^{(i)})$ such that the total producing cost $\sum_{i=1}^{D}c_{t}^{(i)}(x_{t}^{(i)})$ remains low, and the extrinsic given demand $d_{t}$ is balanced. A possible loss function to this regard is:
\begin{equation*}
f_{t}:\real^{D}\rightarrow\real,~ x\mapsto\sum_{i=1}^{D}c_{t}^{(i)}(x^{(i)})+\xi\left(\sum_{i=1}^{D}x^{(i)} -d_{t}\right)^{2}
\end{equation*}

%Consider a power system with $D$ generators. At time slot $t\in\nat$, the goal of economic dispatch is to decide for each $i\in [D]$ the output $x_{t}^{(i)}$ of generator $i$ causing costs $c_{t}^{(i)}(x_{t}^{(i)})$ such that the total generation cost $\sum_{i=1}^{D}c_{t}^{(i)}(x_{t}^{(i)})$ remains low, and the extrinsic given power demand $d_{t}$ is balanced.

In solving the economic dispatch problem, one has to consider several constraints. For instance, the output of each producer $i\in [D]$ can not exceed the value $x^{(i)}_{\max}\in\real_{\geq 0}$ specified e.g. technical restrictions of the producer. Since the violation of this constraint might not be tolerable, Prior works settle the feasible set in the online learning problem formulation as the box-type set $\mathcal{X}=\lrbrace{x\in\real^{D}:~0\leq x^{(i)}\leq x_{\max}^{(i)},~\forall i\in [D]}$. This kind of feasible set is popular in applications (see e.g. \cite{Narayanaswamy2012,Chen2017}). Instead of considering the constraint specified by technical restrictions of the producers, one may instead consider the constraint specified by total output production resulting in the feasible set: 
\begin{equation*}
\mathcal{X}=\lrbrace{x\in\real^{D}_{\geq 0}:~\sum_{i=1}^{D}x^{(i)}\leq B}.
\end{equation*}
In the application of economic dispatch for electrical power, above feasible set corresponds to power transmission restriction specified by the wireline capacity.

%In solving the economic dispatch problem, one has to consider several constraints. For instance, the output of each generator $i\in [D]$ can not exceed the value $x^{(i)}_{\max}\in\real_{\geq 0}$ specified by the technical restriction of the generator. Since the violation of this constraint might not be tolerable, prior works settle the feasible set as the box-type set $\mathcal{X}=\lrbrace{x\in\real^{D}:~0\leq x^{(i)}\leq x_{\max}^{(i)},~\forall i\in [D]}$

Another constraint which one may consider is that the total negative externality (e.g. pollution) $\sum_{i=1}^{D}E^{i\rightarrow j}_{t}(x_{t}^{(i)})$ of a certain kind (e.g. substance) $j\in [N]$ should not exceed a particular value $E^{j}_{\max}$ (e.g. specified by government regulator). In the previous sum, $E^{i\rightarrow j}_{t}$ denotes a function specifying the negative externality of a certain kind $j\in [N]$ at time slot $t$ given a specific output of the producer $i\in [D]$. This gives rise to the penalty function:
\begin{equation}
\label{Eq:aaissidhdjdhdggdgddd}
\begin{split}
g_{t}:\real^{D}\rightarrow\real^{N},~ x\mapsto \left( \sum_{i=1}^{D}E^{i\rightarrow j}_{t}(x^{(i)})-E^{j}_{\max}\right)_{j=1}^{N}
\end{split}
\end{equation}
In the strict sense, it is absurd to think that inter-time compensation of negative externalities occurs. For instance, pollution causes damages irrespective of whether in the earlier time emission constraint is strictly preserved. Rather, the system manager should ensure that $g_{t}(X_{t})$ remains at each time step small. For this reason, the aim of preserving the relaxed constraint given in \eqref{Eq:aoaoshshshsjddgdggdd} seems to be more plausible than the aim of preserving $\sum_{t=1}^{T}g_{t}(X_{t})/T\leq 0$.
%The violation of the constraint $g_{t}(X_{t})\leq 0$ is tolerable however only to some extent. Therefore, it has to be ensured that the violation of this constraint remains small. Moreover,.... 
%Alternatively, $\sum_{i=1}^{D}E^{i\rightarrow j}_{t}(x^{(i)})$ can stand for the total amount of input $j\in [N]$ needed in order to produce $(x^{(i)})_{i=1}^{D}$ at time $t$.
% Instead of considering the constraint specified by technical restriction of the generators, one may consider the constraint specified by electrical power transmission restriction so that for a given wireline capacity $w_{\max}>0$, $\mathcal{X}$ turns to:
%\begin{equation*}
%\mathcal{X}=\lrbrace{x\in\real^{D}_{\geq 0}:~\sum_{i=1}^{D}x^{(i)}\leq w_{\max}},
%\end{equation*}

In order to see where disturbance of the gradient feedback might occur, let us assume that the cost $c_{t}^{(i)}$ is quadratic, i.e.:
\begin{equation*}
c_{t}^{(i)}(x)=a^{(i)}_{t}x^{2}+b^{(i)}_{t}x,
\end{equation*}
where $a^{(i)}_{t}$ and $b^{(i)}_{t}$ are non-negative constants depending on the specific sort of producer. For instance, if the considered commodity is the electrical power and the considered producer is a steam turbine unit, the constants $a^{(i)}_{t}$ and $b^{(i)}_{t}$ depend on the current fuel price, changing over time, and on the maintenance price, including labor price \cite{wood2014}. The first-order information of the cost function can be given explicitly as $\nabla f_{t}(x)=(a^{(i)}_{t}x+b^{(i)}_{t})_{i}$. In reality, one usually has only a disturbed observation of the prices $a^{(i)}_{t}$ and $b^{(i)}_{t}$. For instance, considered the previous instance, the disturbance is due to the uncertainty of the estimate of the current fuel cost. We model this fact by defining the noisy feedback as follows:
\begin{equation*}
\hat{v}_{t}^{(i)}=(a^{(i)}_{t}+\tilde{\xi}^{(i),1}_{t+1})X_{t}+b^{(i)}_{t}+\tilde{\xi}^{(i),2}_{t+1},
\end{equation*}
where $(\tilde{\xi}^{(i),1}_{t})_{t}$ and $(\tilde{\xi}^{(i),2}_{t})_{t}$ are martingale w.r.t. a filtration containing the history of $(X_{t})_{t}$. It holds: $\Erw_{t}[\tilde{\xi}_{t+1}X_{t}]=\Erw_{t}[\tilde{\xi}_{t+1}]X_{t}=0$ and thus by defining $\xi^{(i)}_{t+1}=\tilde{\xi}^{(i),1}_{t+1}X_{t}+\tilde{\xi}^{(i),2}_{t+1}$ it follows that $\hat{v}_{t}=\nabla f_{t}(X_{t})+\xi_{t+1}$, where $(\xi_{t})$ is a martingale. This formulation of $\hat{v}_{t}$ coincides with the model described in Subsection \ref{Subsec:Noisy}.

%Now, in case that the demand has to be fulfilled but there are possibilities for feeding costly additional power from external sources, one may alternatively consider the following objective rewrite $\mathcal{X}_{t}$ and $f_{t}$ as follows:
%\begin{equation}
%\label{Eq:aaissidhdjdhdggdgddd2}
%\begin{split}
%&\mathcal{X}_{t}=\lrbrace{x\in\real^{D}:~\sum_{i=1}^{D}x^{(i)}=d_{t}}\\
%&f_{t}:\real^{D}\rightarrow\real,~ x\mapsto\sum_{i=1}^{D}c_{t}^{(i)}(x^{(i)})+\xi\left(\sum_{i=1}^{D}x_{t}^{(i)} -g_{t}\right)^{2}\\
%\end{split}
%\end{equation} 
%\textcolor{red}{explain $g_{t}$}
\end{example}
\begin{example}[Trajectory Tracking]
Consider a dynamical system:
\begin{equation*}
X_{t+1}=AX_{t}+BU_{t},
\end{equation*}
where $X_{t}$ is the location of a robot and $U_{t}$ is the control action. Let $Y_{t}$ be the location of the target at time slot $t$. The objective of trajectory tracking at time slot $t$ is to choose a control action $U_{t}$ s.t. the tracking error $f_{t}(X_{t})=\norm{X_{t}-Y_{t}}^{2}_{2}/2$ and the smoothness measure $(\beta/2)\norm{X_{t}-X_{t-1}}_{2}^{2}$, where $\beta>0$, is minimized. Possible constraint which one may consider is the energy constraint $\norm{u_{t}}_{2}^{2}\leq u_{2,\max}$ and extremum control value constraints $u_{\min}\leq u_{t}^{(i)}\leq u_{\max}$. Considering a time horizon $T$ we may solve for a given initial states $x_{0}$, the following online problem:
\begin{equation}
\begin{split}
&\min_{(u_{t})_{t=1}^{T}} \sum_{t=1}^{T}\norm{Ax_{t}+B u_{t}-y_{t+1}}_{2}^{2}+\tfrac{\beta}{2}\norm{(A-I)x_{t}+Bu_{t}}_{2}^{2} \\
&\text{s.t. }\sum_{i=1}^{D}(u^{(i)}_{t})^{2}\leq u_{\max,2} \\
&~~~~ u_{\min}^{(i)}\leq u_{t}^{(i)}\leq u_{\max}^{(i)} \quad i\in [D].
\end{split}
\end{equation}
Defining the loss function as:
\begin{equation*}
f_{t}(u)=\norm{Ax_{t}+Bu-y_{t+1}}_{2}^{2}+\tfrac{\beta}{2}\norm{(A-I)x_{t}+Bu}_{2}^{2},
\end{equation*}
we obtain that the loss feedback is given by:
\begin{equation*}
\nabla f_{t}(u)=2 B^{\T}(Ax_{t}-y_{t+1})+(2+\beta) B^{\T}B u+\beta\left[ B^{\T}(A-I)x_{t}\right]. 
\end{equation*}
A possible source of disturbance in the gradient feedback is the location $y_{t+1}$ of the target at time $t+1$. One may also consider the sparsity constraint $\norm{u_{t}}_{1}\leq u_{1,\max}$ instead of/in addition to the energy constraint.	
\end{example}
\subsection{Performance measure and Our Goal}
\label{SubSec:Perf}
In this work, we use the following performance measure, called dynamic regret (see \cite{Zinkevich2003}), which is defined for a sequence of decisions $X_{1},\ldots,X_{t}$ of the online learner as follows:
\begin{equation*}
\Reg^{\text{d}}_{t}:=\sum_{\tau=1}^{t}\left( f_{\tau}(X_{\tau})-f_{\tau}(x_{\tau}^{*})\right)\quad\text{(Dynamic Regret)},
\end{equation*}
where our benchmark is the sequence of the best dynamic solution $\mathbbm{x}^{*}_{t}:=(x^{*}_{\tau})_{\tau\in [t]}$ with:
\begin{equation*}
x^{*}_{\tau}\in\argmin_{x\in\mathcal{X}}f_{\tau}(x)~\text{s.t. }g_{\tau}(x)\leq 0,\quad\forall \tau\in [t].
\end{equation*}
%\begin{remark}
%	Another performance measure is the so called static regret:
%	\begin{equation*}
%	\Reg^{\text{s}}_{t}:=\sum_{\tau=1}^{t} f_{\tau}(\bx_{\tau})-\min_{\overline{\mathcal{Q}}_{t}}\lrbrace{\sum_{\tau=1}^{t}f_{\tau}},
%	\end{equation*}
%	where:
%	\begin{equation*}
%	\overline{\mathcal{Q}}_{t}:=\bigcap_{\tau\in[t]}\mathcal{Q}_{\tau}
%	\end{equation*}
%	Notice that the dynamic regret is always larger than the static regret, i.e. $\Reg^{\text{s}}_{t}\leq\Reg^{\text{d}}_{t}$, since:
%	\begin{equation*}
%	\sum_{\tau=1}^{T}\min_{\mathcal{Q}_{\tau}}f_{\tau}\leq\sum_{\tau=1}^{T}f_{\tau}(\bx),\quad \forall \bx\in\overline{\Q}_{t}.
%	\end{equation*}
%	Thus a bound for dynamic regret gives an upper bound for static regret, which asserts that the dynamic regret is a stronger measure to work with.
%\end{remark}
Throughout this work, we assume that the following regularity condition on the cost function holds:
\begin{assum}
\begin{itemize}
	\item $f_{\tau}$ is convex and subdifferentiable of $\mathcal{X}$.
	\item For each $\tau\in [t]$ and for all $x\in\mathcal{X}$, we have a fix choice of subgradient $\nabla f_{\tau}(x)$ such that $\sup_{x\in\mathcal{X}}\norm{\nabla f_{\tau}(x)}_{*}<\infty$.
\end{itemize}		
\end{assum}

In general $\Reg^{\text{d}}_{t}$ can be negative. This case occurs if for some $\tau\in [t]$, $X_{\tau}$ is not feasible w.r.t. to the constraint $g_{\tau}(x)\leq 0$. However, we have the following lower bound by the mean value Theorem:
\begin{equation}
\label{Eq:aaiisjjsshhdjdhhssssaa}
\Reg^{\text{d}}_{t}\geq-\sum_{\tau=1}^{t}\sup_{\mathcal{X}_{\tau}}\norm{\nabla f_{\tau}}_{*}\norm{\bx_{\tau}-\bx_{\tau}^{*}}\geq -D_{\X}\sum_{\tau=1}^{t}L_{\tau},
\end{equation}
where $L_{\tau}>0$ is a constant fulfilling:
\begin{equation*}
\norm{\nabla f_{\tau}}_{*}\leq L_{\tau}.
\end{equation*}

Related to the dynamic regret, is the following performance measure called dynamic gap defined for $\mathbbm{x}_{t}:=(x_{\tau})_{\tau\in [t]}\subset\mathcal{X}$ as follows:
\begin{equation*}
\Gap_{t}^{\d}(\mathbbm{x}_{t}):=\sum_{\tau=1}^{t}\inn{X_{\tau}-x_{\tau}}{\nabla f_{\tau}(X_{\tau})},
\end{equation*}
which we often use in our analysis. The reason is that besides:
\begin{equation}
\label{Eq:aaiisshhjdhdggshhsss}
\Gap_{t}^{\d}(\mathbbm{x}_{t}^{*})\geq \Reg^{\d}_{t},
\end{equation}
which follows from the convexity of $f_{\tau}$, for all $\tau\in [T]$, the gradients which constitute the building blocks of PDOGA appears in its formulation.

Performance measure for the feasibility of the learner decision respective to the constraint $g_{\tau}(x)\leq 0$, $\tau\in [t]$, which we use in this work is the following:
\begin{equation*}
\hCFit^{r}_{t}:=\sum_{\tau=1}^{t}h(g_{r}(X_{t})).
\end{equation*}
We assume that the following regularity condition on $\hCFit$:
\begin{assum}
	\begin{itemize}
	\item For all $r\in [N]$, $g_{t}^{(r)}$ is convex and (sub-)differentiable on $\mathcal{X}_{t}$.
	\item $h$ is monotonically increasing and sub-differentiable on $\real$.
	\end{itemize}		
\end{assum}
\section{Algorithm Design}
In this section, we provide a novel algorithm which we call generalized online mirror saddle-point (GOMSP) whose aim is to generate online decisions minimizing the performance measures introduced in Subsection \ref{SubSec:Perf}. For convenience, we provide a summary of our finding in Algorithm \ref{Alg:aoaishhjddhhddddeee}.
\begin{algorithm}[h]
	\caption{Generalized Online Mirror Saddle-Point (GOMSP) Method}
	\begin{algorithmic}
		\REQUIRE Time horizon $T\in \nat$, learning rate $\gamma>0$, price sensitivity $\beta>0$, regularization constant $\alpha>0$.
		\REQUIRE Initial score $Y_{1}\in\real^{D}$, - primal iterate $X_{1}=\Phi(Y_{1})$, - dual variable $\Lambda_{1}\in\real_{\geq 0}^{R}$
		\FOR{$t=0,1,2,\ldots,T$}
		\STATE Observe the noisy first-order feedback $$\hat{v}_{t}:=\nabla f_{t} (X_{t})+\xi_{t+1}$$
		\FOR{$r=1,\ldots,R$}
		\STATE Query the first-order $h$-load feedback $\nabla (h\circ g^{r}_{t})(X_{t})$
		\ENDFOR
		\STATE Update the score vector as in \eqref{Eq:Score} 
		\STATE Update the dual variable:
		\begin{equation*}
		\Lambda_{t+1}=\Pi_{\real^{R}_{\geq 0}}\left[(1-\alpha\gamma)\Lambda_{t}+\gamma h(g_{t}(X_{t}))\right] 
		\end{equation*}
		\STATE Update primal variable as in \eqref{Eq:aiaisshsggfgfhddd}:
	
		\ENDFOR
	\end{algorithmic}
	\label{Alg:aoaishhjddhhddddeee}
\end{algorithm}

\subsection{Primal Variable Update - Mirror Descent}
The basis of the primal update of GOMSP is the score vector which is generated from the actual noisy first-order objective - and constraint feedback by the following rule:
\begin{equation}
Y_{t+1}=Y_{t}-\gamma\left(  \hat{v}_{t}+\sum_{r=1}^{N} \left[ \nabla (h\circ g^{(r)}_{t})(X_{t})\right]\Lambda^{(r)}_{t}\right).
\label{Eq:Score}
\end{equation}
The variable $\Lambda^{(r)}_{t}$ is a Lagrange variable that corresponds to the $r$-th constraint, whose update rule will be specified later.

To realize the primal update $X_{t+1}$ from the score vector $Y_{t+1}$ at the time slot $t+1$, we use the so-called mirror map defined in the following:
\begin{definition}[Regularizer and Mirror Map]
	Let $\X$ be a compact convex subset of a Euclidean normed space $(\real^{D},\norm{\cdot})$, and $K>0$. We say $\psi:\X\rightarrow\mathbb{R}$ is a $K$-strongly convex \textit{regularizer} (or also \textit{penalty function}) on $\X$, if $\psi$ is continuous and $K$-strongly convex on $\Z$. 
	%in the sense that for all $x,y\in \mathcal{X}$ and $\lambda\in [0,1]$:
	%\begin{equation*}
	%\psi(\lambda x+(1-\lambda)y)\leq \lambda \psi(x)+(1-\lambda)\psi(y)-\tfrac{K}{2}\lambda (1-\lambda)\norm{x-y}^{2}.
	%\end{equation*}
	The mirror map $\Phi:(\real^{D},\norm{\cdot}_{*})\rightarrow\X$ induced by $\psi$ is defined by:
	%\begin{equation*}
	%\Phi(y):=\argmax\limits_{x\in\mathcal{X}}\left\{\left\langle x,y\right\rangle-\psi(x)\right\}.
	%\end{equation*}
	$$\Phi(y):=\argmax_{x\in\Z}\left\{\left\langle y,x\right\rangle-\psi(x)\right\}$$
\end{definition}
%\begin{equation*}
%\Phi:\real^{D}\mapsto\mathcal{X},\quad y\mapsto\argmax_{x\in\mathcal{X}}~\inn{x}{y}-\psi(x),
%\end{equation*}
%where $\psi:\mathcal{X}\rightarrow \real$ is a $K_{\psi}$-strongly convex function w.r.t. a chosen norm $\norm{\cdot}$. 
Clearly, the mirror map is a generalization of the usual Euclidean projection.
%\begin{example}[Euclidean projection]
%	Let	$\mathcal{Z}$ be compact convex subspace of an Euclidean space. $\psi(x)=\norm{x}^{2}/2$ is a $1$-strongly convex regularizer on $\mathcal{Z}$. Short computation yields that the induced mirror map is the Euclidean projection onto $\mathcal{Z}$, i.e. $\Phi(y)=\argmax_{x\in\mathcal{Z}}\norm{y-x}_{2}$.
%\end{example}
An interesting example of mirror maps is the so-called logit choice $\Phi(y)=\exp(y)/\sum_{l=1}^{D}\exp(y_{l})$ which is generated by the $1$-strongly convex regularizer $\psi(x)=\sum_{k=1}^{D}x_{k}\log x_{k}$ on the probability simplex $\Delta\subset(\real^{D},\norm{\cdot}_{1})$. Other instance of mirror map worth to mentions is $\Phi(Y) = \exp(Y)/(1 + \norm{\exp(Y)}_{1})$ which is defined on the set $\mathcal{X}$ of positive semidefinite matrices $X$ having the nuclear norm $\norm{X}_{1}:=\text{tr}(\abs{X})\leq 1$. The von-Neumann entropy $\psi(X) = \text{tr}(X \log X) + (1 - \text{tr} X) \log(1 - \text{tr}X)$ is a $(1/2)$-strongly convex regularizer on $\Z$ \cite{Kakade2012} (for derivation see e.g. \cite{Merti2017}). 

Having introduced the notion of the mirror map, we can define the primal update rule given a score vector $Y_{t+1}$ and a regularizer $\psi$ as follows:
\begin{equation}
\label{Eq:aiaisshsggfgfhddd}
X_{t+1}=\Phi(Y_{t+1}).
\end{equation}
In case that the chosen regularizer is the Euclidean norm, one can write:
\begin{equation*}
X_{t+1}=\Pi_{\X}\left[ X_{t}-\gamma\left(  \hat{v}_{t}+\sum_{r=1}^{N} \left[ \nabla (h\circ g^{(r)}_{t})(X_{t})\right]\Lambda^{(r)}_{t}\right) \right], 
\end{equation*} 
which is the update rule for the projected noisy gradient descent related to the online Lagrangian:
\begin{equation}
\mathcal{L}_{t}(x,\lambda)=f_{t}(x)+\lambda^{\T}h(g_{t}(x))),\quad x\in\mathcal{X},~\lambda\in\real^{N}_{\geq 0}.\label{Eq:aaussusdggddhdgdgdgdd}
\end{equation} 
This Lagrangian corresponds to the optimization problem:
\begin{equation*}
\min_{x\in\X}f_{t}(x)\quad\text{s.t.}\quad h(g_{t}(x))\leq 0.
\end{equation*}
This observation explains our motivation for defining the primal update as in \eqref{Eq:Score} and \eqref{Eq:aiaisshsggfgfhddd} in case the underlying projection operator is Euclidean.

The reason to use a "projection" mapping, which is in our case the mirror map, more general than Euclidean projection is that it yields a versatile method for the online decision-making process. The mirror map allows us to adapt the first-order penalized iterative method to the geometry of the underlying feasible set of the decision problem and to leverage from the weaker dimension dependency of the algorithm performance. This effect has been recognized earlier in connection with the simple gradient descent method \cite{Nemirovski2008,Nesterov2009}: Using the logit choice instead of Euclidean projection for realizing iterative simple first-order descent method for convex optimization problem on simplex yields a convergence guarantee which depends logarithmically on the - instead of the square root of the underlying dimension. Moreover, using a mirror map other than the Euclidean projection might yield a better dimension dependency of the noise term in the resulted bound since the noise influence is no longer measured by the Euclidean norm.      
%\begin{remark}
%Besides providing a versatile method for the online decision-making process, the advantage of using mirror maps rather than Euclidean projection is a possible much weaker dependence of dimension. For instance, using the logit choice instead of Euclidean projection for realizing iterative simple first-order descent method for convex optimization problem on a simplex yields a convergence guarantee which depends logarithmically on the - instead of the square root of the underlying dimension (see e.g. \cite{Nemirovski2008}).
%\end{remark}      

Another factor that is variable in the update rule of GOMSP is the function $h$. In this regard, we provide for the convenience of the reader a particular form of \eqref{Eq:Score} in the following: 
\begin{example}
The function $h$ which we mainly have in mind is $h(\cdot)=[\cdot]^{p}_{+}$. By choosing the subgradient as follows:
\begin{equation*}
\nabla (h\circ g^{r}_{t})(x)=
\begin{cases}
p~(g_{t}^{r}(x))^{p-1} \nabla g_{t}^{(r)}(x)&\quad \text{if }g_{t}^{(r)}(x)\geq 0\\
0&\quad \text{else},
\end{cases}
\end{equation*}
we may write:
\begin{equation*}
Y_{t+1}=Y_{t}-\gamma\left(  \hat{v}_{t}+p\sum_{r\in\mathcal{A}_{t}} (g_{t}^{r}(X_{t}))^{p-1} \nabla g_{t}^{r}(X_{t})\Lambda^{(r)}_{t}\right),
\label{Eq:Score2}
\end{equation*}
where $\mathcal{A}_{t}=\lrbrace{r\in [R]:~g^{r}_{t}(x)> 0}$ denotes the set of active constraints at time $t$.
\end{example}

\subsection{Dual variable update}
The primary role of the dual variable $\Lambda_{t}$ is to provide the primal variable information about the actual amount of the constraint violation. One might draw the analogy between this variable and the prices in markets whose role is to signal the participants to what extent the corresponding resources are scarce. In particular, $\Lambda_{t}$ has to reflect the actual constraint violation state. Besides, another crucial requirement for the dual variable is that it does not grow unboundedly. Otherwise, the constraint term in the primal update overthrow the cost part and consequently the primal update concentrates on reducing the amount of violation rather than minimizing the regret.
 
In hindsight of those aspects, we give the update rule the dual variable of GOMSP as follows:
\begin{equation}
\label{Eq:Dual}
\Lambda_{t+1}=\Pi_{\real^{R}_{\geq 0}}\left[(1-\alpha\gamma)\Lambda_{t}+\gamma h(g_{t}(X_{t}))\right]. 
\end{equation}
In case that $\alpha=0$, \eqref{Eq:Dual} turn to the simple dual gradient ascent corresponds to the Lagrangian \eqref{Eq:aaussusdggddhdgdgdgdd}. The idea behind adding the regularization term $\alpha\gamma \Lambda_{t}$ is to reduce the growth of the dual variable by decaying the influence of previous constraint states: To see this, notice that if $h\geq 0$, we can omit the projection operator in the expression \eqref{Eq:Dual}. Consequently:
\begin{equation*}
\Lambda_{t+1}=(1-\alpha\gamma)^{t}\Lambda_{1}+\gamma\sum_{\tau=1}^{t}(1-\alpha\gamma)^{\tau-t}h(g_{\tau}(X_{\tau})).
\end{equation*}
If $\Lambda_{1}=0$, we have:
\begin{equation*}
\Lambda_{t+1}=\gamma\sum_{\tau=1}^{t}(1-\alpha\gamma)^{t-\tau}h(g_{\tau}(X_{\tau})).
\end{equation*}
Thus the influence of the $\tau$-th constraint function term to the dual variable at time $t+1$ decays with $\exp(-\tau a)$ where $a=-\ln(1-\alpha\gamma)$. This can be advantageous in the online environment since $g_{\tau}$ for different time-slots (with large distance) not necessarily correlate.

Another reason for defining \eqref{Eq:Dual} is that the resulted dual dynamic gives rise about the cumulative constraint state in the following sense:
\begin{lemma}[From Dual Dynamic to Constraint Violation]
	\label{Lem:aauissshdhdhhdjsss}
	Suppose that $\Lambda_{1}=0$. It holds:
	\begin{equation*}
	\hCFit_{t}^{r}\leq \tfrac{\norm{\Lambda_{t+1}}_{2}}{\gamma}+\alpha\sum_{\tau=1}^{t}\norm{\Lambda_{\tau}}_{2}
	\end{equation*}
\end{lemma}
\begin{proof}
	By \eqref{Eq:Dual}, we have $\Lambda^{r}_{\tau+1}\geq\Lambda^{r}_{\tau}+\gamma h(g^{r}_{\tau}(X_{\tau}))-\alpha\gamma\Lambda^{r}_{\tau}$.
%	\begin{align*}
%	\Lambda^{r}_{\tau+1}\geq\Lambda^{r}_{\tau}+\gamma h(g^{r}_{\tau}(X_{\tau}))-\alpha\gamma\Lambda^{r}_{\tau}.
%	\end{align*}
	So summing, telescoping, and the assumption $\Lambda_{1}=0$ give:
	\begin{equation}
	\label{Eq:aasishdhhhdd}
	%\begin{split}
	\gamma\sum_{\tau=1}^{t}h(g^{r}_{\tau}(X_{\tau}))
%	&\leq \Lambda_{t+1}^{r}-\Lambda_{1}^{r}+\alpha\gamma\sum_{\tau=1}^{t}\Lambda_{\tau}^{r}\\
	%&
	\leq \Lambda_{t+1}^{r}+\alpha\gamma\sum_{\tau=1}^{t}\Lambda_{\tau}^{r}
	%\end{split}
	\end{equation}
	By the fact that $\Lambda^{r}_{t}\geq 0$, it holds $\Lambda_{t}^{r}\leq \norm{\Lambda_{t}}_{2}$. Finally, the latter and \eqref{Eq:aasishdhhhdd} give the desired inequality.
%	:
%	\begin{equation*}
%	\Lambda_{t}^{r}\leq\sqrt{\sum_{r=1}^{R}(\Lambda_{t}^{r})^{2}}=\norm{\Lambda_{t}}_{2}.
%	\end{equation*}
%	This asserts:
%	\begin{equation*}
%	\gamma\sum_{\tau=1}^{t}h(g^{r}_{\tau}(X_{\tau}))\leq \norm{\Lambda_{t+1}}_{2}+\alpha\gamma\sum_{\tau=1}^{t}\norm{\Lambda_{\tau}}_{2}.
%	\end{equation*}
\end{proof}
%\subsection{Relation to other methods}
%Notice that MOSP is a special case of GOMSP with no dual regulation, i.e. $\alpha=0$, $h$ is the identity function, and the Euclidean norm as the regularizer. 
%\subsection{Decentralized implementation with Idle Price}
%\textcolor{red}{
%\begin{itemize}
%	\item Less access to the first order oracle of penalty function.
%	\item Less information about the actual constraint state
%\end{itemize}
%}
\section{Performance Analysis}
\label{Sec:Perf}
To analyze the performance of the algorithm, we leverage from Lyapunov-type argumentation. In doing that, we use as energy functions both, the distance between the iterate $Y_{t}$ of the algorithm and the current constraint minimizer of the cost function and the norm of the dual variable. 

This sort of Lyapunov function is standard \cite{Zinkevich2003} besides the fact that we use the Fenchel coupling as the primal iterate distance function defined as follows:
\begin{definition}[Fenchel Coupling]
	Let $\psi:\X\rightarrow\real$ be a penalty function on a compact convex subset $\X$ of a Euclidean normed space $(\real^{D},\norm{\cdot})$. The Fenchel coupling induced by $\psi$ is defined as $F:\mathcal{X}\times(\real^{D},\norm{\cdot}_{*})\rightarrow\real_{\geq 0}$ given by:
	$$ F(x,y):=\psi(x)+\psi^{*}(y)-\inn{y}{x}$$
	%\begin{equation*}
	%F(p,y)=\psi(p)+\psi^{*}(y)-\inn{p}{y},\quad p\in \X,~y\in E^{*}.
	%\end{equation*}
\end{definition} 
As the Lyapunov function, we use specifically:
\begin{equation*}
\mathcal{E}_{t}(x)=\underbrace{F(x,Y_{t})}_{=:\mathcal{E}^{1}_{t}(x)}+\underbrace{\tfrac{\norm{\Lambda_{t}}_{2}}{2}}_{=:\mathcal{E}_{t}^{2}}.
\end{equation*}
%\subsection{Upper Bound}
To analyze the performance of the proposed algorithm, we give in the following an upper bound for the primal dynamic and dual dynamic.
\subsection{Lyapunov Analysis}
\paragraph*{Primal Dynamic}
For convenience, we rewrite \eqref{Eq:Score} as:
\begin{equation*}
\begin{split}
Y_{t+1}&=Y_{t}-\gamma\left( \hat{v}_{t}- \left[ \nabla (h\circ g_{t})(X_{t})\right]^{\T}\Lambda_{t}\right),
\end{split}
\end{equation*}
where:
\begin{equation*}
\left[ \nabla (h\circ g_{t})(X_{t})\right]^{\T} =\left[\nabla (h\circ g^{1}_{t})(X_{t}),\ldots,\nabla (h\circ g^{R}_{t})(X_{t}) \right] 
\end{equation*}
The following result gives the upper bound of the one-step difference $\Delta\mathcal{E}^{1}_{t}(x):=\mathcal{E}^{1}_{t+1}(x)-\mathcal{E}^{1}_{t}(x)$:
\begin{lemma}
	\label{Lem:aaiisshdggdhgdhsssss}
	For any $x\in\mathcal{X}$:
\begin{align*}
\Delta\mathcal{E}^{1}_{t}(x)
%&\leq\gamma_{k}\inni{X_{k}-x}{v(X_{k})-[\nabla g(X_{k})]^{T}\Lambda_{k}}\\
%&+\gamma_{k}\inni{X_{k}-x}{\xi_{k+1}}\\
%&+\frac{\gamma_{k}^{2}}{K}\left(\n1ormi{[\nabla g(X_{k})]^{T}\Lambda_{k}}_{*}^{2}+\normi{v(X_{k})+\xi_{k+1}}_{*}^{2}\right)\\
\leq&-\gamma\inn{X_{t}-x}{\nabla f_{t}(X_{t})}\\
&-\gamma\inn{X_{t}-x}{[\nabla (h\circ g_{t})(X_{t})]^{\T}\Lambda_{t}}\\&
+\gamma\tilde{\xi}_{t+1}+\tfrac{\gamma^{2}C_{1,\psi}^{2}}{K}\norm{\Lambda_{t}}_{2}^{2}+\tfrac{2\gamma^{2}}{K}(C_{2,\psi}^{2}+\norm{\xi_{t+1}}_{*}^{2}),
\end{align*}
where $C_{1,\psi},C_{2,\psi}$, are the smallest  constants $C_{1},C_2>0$ satisfying for all $\lambda\in\real^{R}_{\geq 0}$ and $x\in\mathcal{X}$:
\begin{equation*}
\norm{[\nabla (h\circ g_{t})(x)]^{\T}\lambda}_{*}\leq C_{1}\norm{\lambda}_{2}\quad\norm{\nabla f_{t}(x)}_{*}\leq C_{2}
\end{equation*}
\end{lemma}
\paragraph*{Dual Dynamic}
The expression given in Lemma \ref{Lem:aaiisshdggdhgdhsssss} possesses a dependency on the dual variable $\Lambda_{t}$. So to continue, it stands clear to analyze the dynamic of this variable. Toward this direction, we have the following result on the drift of the Lagrangian:
\begin{lemma}
	\label{Lem:aaiisshdggdhgdhsssss2}
	For $\tilde{x}\in\mathcal{Q}_{\tau}$:
	\begin{align*}
	\Delta\mathcal{E}_{\tau}^{(2)}\leq&\gamma \inn{\left[ \nabla (h\circ g_{\tau})(X_{\tau})\right]^{\T} \Lambda_{\tau}}{X_{\tau}-\tilde{x} }\\
	&-(\alpha\gamma-\alpha^{2}\gamma^{2})\norm{\Lambda_{\tau}}_{2}^{2}+\gamma^{2}C_{3}^{2},\nonumber
	\end{align*}
	where $C_{3}>0$ is a constant satisfying:
	\begin{equation}
	\label{Eq:asjsjskdhdggddfsssddd}
	\norm{h(g(x))}_{2}\leq C_{3},\quad\forall x\in\mathcal{X}.
	\end{equation}
\end{lemma} 
For ease of the readibility, we provide the proof of this Lemma in Appendix \ref{Subsec: AppProof}.
\paragraph*{Primal-Dual Dynamic}
By combining previous auxiliary statements on the dynamic of the primal - and dual variable, we obtain the following result:
\begin{theorem}
	\label{Thm:aapsjfgezzegfddd}
Suppose that:
\begin{equation}
\label{Eq:aaiisshsdggdgdfdffdddffff}
\alpha-\gamma(\alpha^{2}-\tfrac{ C_{1}^{2}}{K})\geq 0
\end{equation}
For $\mathbbm{x}_{t}:=(x_{\tau})_{\tau\in [t]}\subset\mathcal{X}$ with $x_{\tau}\in\mathcal{Q}_{\tau}$ for all $\tau\in [t]$:
%\begin{align*}
%&\mathcal{E}_{t}^{1}(\mathbbm{u}_{t})+\mathcal{E}_{t}^{2}\\
%&\leq -\gamma\Gap(\mathbbm{u}_{t})+\gamma S_{t}(\mathbbm{u}_{t})+\tfrac{2\gamma^{2}}{K} t C_{2}^{2}+\tfrac{2\gamma^{2}}{K}R_{t}+t\gamma^{2} C_{3}^{2},
%\end{align*}
	\begin{align*}
	&\Gap^{\d}_{t}(\mathbbm{x}_{t})+\tfrac{\norm{\Lambda_{t+1}}_{2}^{2}}{2\gamma}\leq -\tfrac{\mathcal{V}_{t}(\mathbbm{x}_{t})}{\gamma}\\
	&+\tfrac{\norm{\Lambda_{1}}_{2}^{2}}{2\gamma}+t\gamma C^{2}_{\psi}+S_{t}(\mathbbm{x}_{t})+\tfrac{2\gamma}{K}R_{t},
	\end{align*}
where:
\begin{align*}
S_{t}(\mathbbm{u}_{t})&=\sum_{\tau=1}^{t}\tilde{\xi}_{\tau+1}(u_{\tau}),\quad R_{t}=\sum_{\tau=1}^{t}\norm{\xi_{\tau+1}}_{*}^{2},\\
&\mathcal{V}^{1}_{t}(\mathbbm{x}_{t}):=\sum_{\tau=1}^{t}\Delta\mathcal{E}_{\tau}^{1}(x_{\tau}),\quad 
C_{\psi}^{2}:=\tfrac{2C_{2,\psi}^{2}}{K}+C_{3,\psi}^{2},
\end{align*}
\end{theorem}
\begin{proof}
From Lemma \ref{Lem:aaiisshdggdhgdhsssss} and Lemma \ref{Lem:aaiisshdggdhgdhsssss2}, we obtain for any $\tilde{x}\in\mathcal{Q}_{t}$:
\begin{align*}
&\Delta\mathcal{E}_{t}^{1}(\tilde{x})+\Delta\mathcal{E}_{t}^{2}\\
&\leq -\gamma\inn{X_{t}-\tilde{x}}{\nabla f_{t}(X_{t})}+\gamma \tilde{\xi}_{t+1}+\tfrac{2\gamma^{2}}{K}(C_{2}^{2}+\norm{\xi_{t+1}}^{2}_{*})\\
&\underbrace{-\gamma(\alpha-\gamma\alpha^{2}-\tfrac{\gamma C_{1}^{2}}{K})\norm{\Lambda_{t}}_{2}^{2}}_{=:e_{1}}+\gamma^{2} C_{3}^{2}.
\end{align*}
The condition \eqref{Eq:aaiisshsdggdgdfdffdddffff} help us to get rid of the expression $e1$, which involves the dual variable. By summing the resulted inequality and since $2\sum_{\tau=1}^{t}\Delta\mathcal{E}_{t}^{2}=\norm{\Lambda_{t+1}}^{2}_{2}-\norm{\Lambda_{1}}^{2}_{2}$, we obtain the desired statement.
\end{proof}
%\subsection{Lower Bound For }
By the relation \eqref{Eq:aaiisshhjdhdggshhsss}, $\Gap_{t}^{\d}(\mathbbm{x}^{*}_{t})$ gives rise to the dynamic regret. Moreover, Lemma \ref{Lem:aauissshdhdhhdjsss} asserts that the Lagrangian variable contains the information about the cumulation of the constraint violation. Thus we come closer to achieving the objective of providing performance guarantee for the proposed algorithm. As usual, the terms $S_{t}$ and $R_{t}$ due to objective feedback noise can be handled by taking the expectation. So, the only term at which a closer look should be taken is $\mathcal{V}^{1}_{t}(\mathbbm{x}_{t}^{*})$.

\paragraph*{Lower bound for Primal Energy Function}
In case that that the environment is not adversary, i.e., $f_{\tau}$ remains for all $\tau\in [t]$ the same, it holds by telescoping: $$\mathcal{V}_{t}^{1}(\mathbbm{x}_{t}^{*})=F(x^{*},Y_{t+1})-F(x^{*},Y_{1})\geq -F(x^{*},Y_{1}),$$ where $x^{*}$ denotes the constrained minimizer of $f_{\tau}$. What we may do in the adversary case is to interpolate $\mathcal{V}_{t}^{1}(\mathbbm{x}_{t}^{*})$ by the cumulative difference of the benchmark sequence $\mathbbm{x}_{t}$. In order to execute this procedure, we assume the following:
\begin{assum}
The regularizer is nowhere steep in the sense that $\psi$ is differentiable on $\X$. 
%
%$\nabla\psi$ is uniformly bounded on the relative interior of $\X$, i.e.:
%\begin{equation*}
%\sup_{x\in\textnormal{relint}(\X)}\norm{\nabla \psi(x)}_{*}<\infty.
%\end{equation*}	
%[Better: $\psi$ is differentiable on $\X$.]
\end{assum}
Before we proceed, we first discuss this assumption in the following:
\begin{remark}
Suppose that $\mathcal{X}=\lrbrace{x\in\real^{D}_{\geq 0}:~\sum_{i=1}^{D}x_{i}\leq B}$ for a fixed constant $B>0$. The Euclidean norm seen as a regularizer on $\X$ is clearly nowhere steep. In contrast to the Euclidean norm, the entropy function $\psi(x)=\sum_{i=1}^{D}x_{i}\ln(x_{i})$ as a regularizer is not nowhere steep since the gradient of $\psi$ grows unboundedly as the argument goes to the element of $\mathcal{X}$ which possesses zero coordinates. However, we may instead use the smoothed entropy $\psi_{\epsilon}(x)=\psi(x+\epsilon)$ where $\epsilon>0$ is a chosen constant. As we will discuss later This procedure does not have any significant impact on the dynamic of our algorithm.
\end{remark}
	
We first show that is the regularizer is nowhere steep then the Fenchel coupling is Lipschitz in the first argument:
\begin{lemma}
\label{Lem:aaususggdhhdgdd}
Suppose that $\psi$ is nowhere steep. Then for all $x_{1},x_{2}\in\X$ and $y\in \real^{D}$ :
\begin{equation*}
\abs{F_{\psi}(x_{1},y)-F_{\psi}(x_{2},y)}\leq 2L_{\psi}\norm{x_{1}-x_{2}},
\end{equation*}
where $L_{\psi}>0$ is given by:
\begin{equation}
\label{Eq:aiashjshsshgdgddggdhddd}
 L_{\psi}:=\sup_{x\in\X}\norm{\nabla\psi(x)}_{*}.
\end{equation}
\end{lemma}
\begin{proof}
By definition of $F$ and the triangle inequality, we have:
\begin{equation*}
\begin{split}
\abs{F(x_{1},y)-F(x_{2},y)}&\leq \abs{\psi(x_{1})-\psi(x_{2})}+\abs{\inn{y}{x_{1}-x_{2}}}
%\\
%&\leq L_{\psi}\norm{x_{1}-x_{2}}+\abs{\inn{y}{x_{1}-x_{2}}},
\end{split}
\end{equation*}
Mean value Theorem and the nowhere-steepness of $\psi$ asserts:
\begin{equation*}
\abs{\psi(x_{1})-\psi(x_{2})}\leq L_{\psi}\norm{x_{1}-x_{2}}.
\end{equation*}
Now, since $\psi$ is nowhere steep, it follows from Proposition \ref{Prop:aiaishshjfggfhdhddd} that $\Phi$ is surjective. So we can find a $x\in\mathcal{X}$ s.t. $x=\Phi(y)$ and thus (again by Proposition \ref{Prop:aiaishshjfggfhdhddd}) $y=\nabla \psi(x)$. Consequently we have by H\"older inequality:
\begin{equation*}
\abs{\inn{y}{x_{1}-x_{2}}}\leq \norm{\nabla\psi(x)}_{*}\norm{x_{1}-x_{2}}\leq L_{\psi}\norm{x_{1}-x_{2}}. 
\end{equation*}
\end{proof}
We are now ready to give a lower bound for $\mathcal{V}^{1}_{t}$:
\begin{lemma}
Suppose that $\psi$ is nowhere steep. It holds:
\begin{equation*}
\begin{split}
\mathcal{V}_{t}^{(1)}(\mathbbm{x}_{t})\geq-F(x_{1},Y_{1})-L_{\psi}\Vari^{\psi}(\mathbbm{x}_{t+1}),
\end{split}
\end{equation*}
where $L_{\psi}$ is given in \eqref{Eq:aiashjshsshgdgddggdhddd}:
\begin{equation*}
\Vari^{\psi}(\mathbbm{x}_{t+1}):=\sum_{\tau=1}^{t}\norm{\bx_{\tau+1}^{*}-\bx_{\tau}^{*}}
\end{equation*}
\end{lemma}
\begin{proof}
Applying Lemma \ref{Lem:aaususggdhhdgdd}, we obtain:
\begin{equation*}
\begin{split}
&\Delta\mathcal{E}_{\tau}^{(1)}(x_{\tau})=F(x_{\tau},Y_{\tau+1})-F(x_{\tau},Y_{\tau})\\
&=
F(x_{\tau+1},Y_{\tau+1})-F(x_{\tau},Y_{\tau})\\
&~~~+F(x_{\tau},Y_{\tau+1})-F(x_{\tau+1},Y_{\tau+1})\\
&\geq F(x_{\tau+1},Y_{\tau+1})-F(x_{\tau},Y_{\tau})- L_{\psi}\norm{x_{\tau+1}-x_{\tau}}
\end{split}
\end{equation*}
Thus summing over $\tau\in [t]$:
\begin{equation*}
\begin{split}
\mathcal{V}_{t}^{(1)}(\mathbbm{x}_{t+1})&\geq F(x_{t+1},Y_{t+1})-F(x_{1},Y_{1})-L_{\psi}\Vari_{t}(\mathbbm{x}_{t+1})\\
&\geq-F(x_{1},Y_{1})-L_{\psi}\Vari_{t}(\mathbbm{x}_{t+1})
\end{split}
\end{equation*}
\end{proof}
\subsection{Dynamic Regret bound}
\begin{theorem}
	\label{Thm:aaoosjshdggdhhdggshhsss}
	Suppose that $\psi$ is nowhere steep and that \eqref{Eq:aaiisshsdggdgdfdffdddffff} is fulfilled.
For $\mathbbm{u}_{t}:=(u_{\tau})_{\tau\in [t]}\subset\mathcal{X}$ with $u_{\tau}\in\mathcal{Q}_{\tau}$ for all $\tau\in [t]$:	
\begin{equation*}
\begin{split}
\Erw[\Gap^{\d}_{t}(\mathbbm{x}_{t})]&\leq \tfrac{F(x_{1},Y_{1})}{\gamma}+L_{\psi}\tfrac{\mathbb{V}(\mathbbm{x}_{t})}{\gamma}+\tfrac{\norm{\Lambda_{1}}_{2}^{2}}{2\gamma}+t\gamma C_{\psi}^{2}\\
&+\tfrac{2\gamma}{K}\sum_{\tau=1}^{t}\sigma_{t+1}^{2},
\end{split}
\end{equation*}
where:
\begin{equation*}
C_{\psi}^{2}:=\tfrac{2C_{2,\psi}^{2}}{K}+C_{3,\psi}^{2},
\end{equation*}
and:
\begin{equation*}
\Erw[\norm{\xi_{\tau}}_{*}^{2}]\leq \sigma_{\tau}^{2}.
\end{equation*}
\end{theorem}
\begin{proof}
First notice that $\norm{\Lambda_{t}}_{2}\geq 0$.
Combining this with \eqref{Eq:aaiisshsdggdgdfdffdddffff}, it holds:	
\begin{equation}
\label{Eq:aausushdgdggdhhssssdd}
\begin{split}
\Gap^{\d}_{t}(\mathbbm{x}_{t})&\leq \tfrac{F(x_{1}^{*},Y_{1})}{\gamma}+\tfrac{L_{\psi}\mathbb{V}_{t}}{\gamma}+\tfrac{\norm{\Lambda_{1}}_{2}^{2}}{2\gamma}+\gamma t C_{\psi}^{2}\\
&+S_{t}(\mathbbm{x}_{t})+\tfrac{2\gamma}{K}R_{t}.
\end{split}
\end{equation}
Now, one can check that $S_{t}(\mathbbm{x}_{t})$, $t\in \nat$ is a martingale. Consequently $\Erw[S_{t}(\mathbbm{x}_{t})]=\Erw[\inn{X_{1}-x_{1}}{\xi_{2}}]=\Erw[\inn{X_{1}-x_{1}}{\Erw[\xi_{2}|\mathcal{F}_{1}]}]=0$. So taking the expectation over \eqref{Eq:aausushdgdggdhhssssdd}, we obtain the desired statement. 

\end{proof}
\begin{corollary}
Suppose that the requirements of Theorem \ref{Thm:aaoosjshdggdhhdggshhsss} are fulfilled and suppose that in addition $Y_{1}=0$ and $\Lambda_{0}=0$. Moreover suppose that the noise is persistent in the sense that there exists $\sigma>0$ s.t. $\Erw[\norm{\xi_{\tau}}_{*}^{2}]\leq \sigma^{2}$ for all $\tau$. With:
\begin{equation*}
\gamma=\Theta(T^{-1/2}),
\end{equation*}
it holds:
\begin{equation*}
\Erw[\Reg_{t}^{\d}]\leq\left[ \mathcal{D}(\X,\psi)+L_{\psi}\mathbb{V}_{t}\right]\mathcal{O}(\sqrt{T})+ (C_{\psi}^{2}+\tfrac{\sigma^{2}}{K})t\mathcal{O}(T^{-1/2})
\end{equation*}
where:
\begin{equation*}
\mathcal{D}(\X,\psi)=\sup_{x\in\X}\psi(x)-\inf_{x\in\X}\psi(x)\quad C_{\psi}^{2}=C_{2,\psi}^{2}+C_{3,\psi}^{2}
\end{equation*}
\end{corollary}
\begin{proof}
By the relation \eqref{Eq:aaiisshhjdhdggshhsss}, it follows that the upper bound for the gap given 	By the assumption $Y_{1}=0$, it holds:
\begin{equation*}
F(x_{1},Y_{1})=\psi(x_{1})-\psi^{*}(0)=\psi(x_{1})-\inf_{x\in\X}\psi(x)\leq\mathcal{D}(\X,\psi)
\end{equation*}
Previous observations and the assumption $\Lambda_{1}=0$ and the asummption that the noise is persistent yields:
\begin{equation*}
\begin{split}
\Erw[\Reg^{\d}_{t}]&\leq \tfrac{\mathcal{D}(\X,\psi)+L_{\psi}\mathbb{V}{t}}{\gamma}+\left(  C_{\psi}^{2}+\tfrac{2\sigma^{2}}{K}\right)t\gamma ,
\end{split}
\end{equation*}
\end{proof}
So from above result, we have that $\mathbb{E}[\Reg_{T}^{\d}]$ is of order $\mathcal{O}((1+\mathbb{V}_{t}+\sigma)\sqrt{T})$ in case that the online environment changes slowly in the sense that $\mathbb{V}_{T}\leq\mathcal{O}(T^{p})$ where $p<1/2$, the expected regret is sublinear.
\subsection{Constraint Violation Analysis}
Requirements:
\begin{equation}
\label{Eq:aaoisoshdhdggdhdd}
\norm{\nabla f_{t}(x)}_{*}\leq L_{f},\quad\forall x\in\mathcal{X},t
\end{equation}
\begin{equation*}
\bigcap_{\tau\in [t]}\mathcal{Q}_{\tau}\neq \emptyset.
\end{equation*}
\eqref{Eq:aaoisoshdhdggdhdd} asserts that for $x\in\mathcal{X}$:
\begin{equation*}
-\inn{X_{t}-x}{\nabla f_{t}(X_{t})}\leq\norm{X_{t}-x}\norm{\partial f_{t}(X_{t})}_{*}\leq D_{\X} L_{f}
\end{equation*}

\begin{theorem}
	For any  $\tilde{x}\in\bigcap_{\tau\in [t]}\lrbrace{g_{\tau}\leq 0}$, it holds:
%\begin{align*}
%\tfrac{\norm{\Lambda_{t+1}}_{2}^{2}}{2}\leq&\tfrac{\norm{\Lambda_{1}}_{2}^{2}}{2}+\tfrac{F(\tilde{x},Y_{1})}{\beta} +\tfrac{\gamma t D_{\X} L_{f}}{\beta}+\tfrac{2\gamma^{2}t C_{2}^{2}}{\beta K} +t C_{3}^{2}\\
%&+\tfrac{\gamma S_{t}(\tilde{x})}{\beta}+\tfrac{2\gamma^{2}R_{t}}{\beta K}
%\end{align*}
\begin{align*}
\tfrac{\Erw[\norm{\Lambda_{t+1}}_{2}^{2}]}{2}
&\leq\gamma t \mathcal{D}_{\X}L_{f}+F(\tilde{x},Y_{1})+\tfrac{\norm{\Lambda_{1}}_{2}^{2}}{2}+t\gamma^{2} C_{\psi}^{2}\\
&+\tfrac{2\gamma^{2}\sum_{\tau=1}^{t}\sigma_{\tau+1}^{2}}{K}
\end{align*}
\end{theorem}
\begin{proof}
We have for any $\mathbbm{x}_{t}\subset\mathcal{X}$:
\begin{equation*}
-\Gap^{\d}_{t}(\mathbbm{x}_{t})=-\sum_{\tau=1}^{t}\inn{X_{\tau}-x_{\tau}}{\nabla f_{\tau}(X_{\tau})}\leq t D_{\X} L_{f},
\end{equation*}
and for $\mathbbm{x}_{t}\subset\mathcal{X}$ with $x_{\tau}=\tilde{x}\in\mathcal{X}$ for all $\tau$:
\begin{equation*}
\mathcal{V}_{t}^{1}(\mathbbm{x}_{t})=F(\tilde{x},Y_{t+1})-F(\tilde{x},Y_{1})\geq -F(\tilde{x},Y_{1})
\end{equation*}	

Combining this with Theorem 	\ref{Thm:aapsjfgezzegfddd}, it holds for $\mathbbm{x}_{t}\subset\mathcal{X}$ with $x_{\tau}=\tilde{x}\in\bigcap_{\tau}\mathcal{Q}_{\tau}$ for all $\tau$:
\begin{align*}
\tfrac{\Erw[\norm{\Lambda_{t+1}}_{2}^{2}]}{2}&\leq-\gamma\Erw[\Gap^{\d}_{t}(\mathbbm{x}_{t})]+F(\tilde{x},Y_{1})\\
&+\tfrac{\norm{\Lambda_{1}}_{2}^{2}}{2}+t\gamma^{2} C_{\psi}^{2}+\tfrac{2\gamma^{2}\sum_{\tau=1}^{t}\sigma_{\tau+1}^{2}}{K},\\
&\leq\gamma t \mathcal{D}_{\X}L_{f}+F(\tilde{x},Y_{1})\\
&+\tfrac{\norm{\Lambda_{1}}_{2}^{2}}{2}+t\gamma^{2} C_{\psi}^{2}+\tfrac{2\gamma^{2}\sum_{\tau=1}^{t}\sigma_{\tau+1}^{2}}{K}
\end{align*}
\end{proof}
\begin{corollary}
	Suppose that $Y_{1}=0$ and $\Lambda_{1}=0$. For $\gamma=\Theta(T^{-1/2})$ and $\alpha=\Theta(T^{-1/2})$ fulfilling \eqref{Eq:aaiisshsdggdgdfdffdddffff},  
%	\begin{itemize}
%		\item $Y_{1}=0$ and $\Lambda_{1}=0$
%		\item $\gamma=\Theta(T^{-1/2})$ and $\alpha=\Theta(T^{-1/2})$
%		\item Tractability condition 
%	\end{itemize}
It holds:
\begin{equation*}
\begin{split}
\hCFit_{t}^{r}\leq&\sqrt{\mathcal{D}(\X,\psi)}\mathcal{O}(T^{1/2})+\sqrt{D_{\X}  L_{f}}\mathcal{O}(T^{3/4}) \\
&+(C_{\psi}^{2}+\tfrac{2\sigma^{2}}{K})^{1/2}\mathcal{O}(T^{1/2})
\end{split}
\end{equation*}
\end{corollary}
\begin{proof}
By the assumption $Y_{1}=0$, we have $F(\tilde{x},Y_{1})\geq \mathcal{D}(\X,\psi)$. So, it holds:
\begin{align*}
\tfrac{\Erw[\norm{\Lambda_{t+1}}_{2}^{2}]}{2}\leq&\mathcal{D}_{\X}L_{f}t\mathcal{O}(T^{-1/2})+\mathcal{D}(\X,\psi)\\
&+(C^{2}+\tfrac{2\sigma^{2}}{K})t\mathcal{O}(T^{-1})
\end{align*}
Consequently by Jensen's inequality:
\begin{align*}
\Erw[\norm{\Lambda_{t+1}}_{2}]\leq&\sqrt{\mathcal{D}(\X,\psi)} + \sqrt{D_{\X} L_{f}} t^{1/2}\mathcal{O}(T^{-1/4})\\&+(C_{\psi}^{2}+\tfrac{2\sigma^{2}}{K})^{1/2}t^{1/2}\mathcal{O}(T^{-1/2}).
\end{align*}
Consequently:
\begin{align*}
\tfrac{\Erw[\norm{\Lambda_{t+1}}_{2}]}{\gamma}&\leq\sqrt{D_{\X}L_{f}} t^{1/2}\mathcal{O}(T^{1/4})+\sqrt{\D(\X,\psi)}\mathcal{O}(T^{1/2})\\
&+(C_{\psi}^{2}+\tfrac{2\sigma^{2}}{K})^{1/2}t^{1/2}\mathcal{O}(1)
\end{align*}
Now, we have:
\begin{equation*}
\begin{split}
\alpha\sum_{\tau=1}^{t}\Erw[\norm{\Lambda_{\tau}}_{2}]&\leq\sqrt{\mathcal{D}(\X,\psi)}t\mathcal{O}(T^{-1/2})\\
&+\sqrt{D_{\X} L_{f}} \mathcal{O}(t^{3/2})\mathcal{O}(T^{-3/4})\\
&+(C_{\psi}^{2}+\tfrac{2\sigma^{2}}{K})^{1/2} \mathcal{O}(t^{3/2})\mathcal{O}(T^{-1}).
\end{split}
\end{equation*}
Consequently:
\begin{equation*}
\begin{split}
&\hCFit_{t}^{r}\leq\sqrt{\mathcal{D}(\X,\psi)}\left[\mathcal{O}(\sqrt{T})+t\mathcal{O}(T^{-1/2}) \right] \\
&+\sqrt{D_{\X}  L_{f}}\left[\sqrt{t}\mathcal{O}(T^{1/4})+\mathcal{O}(t^{3/2})\mathcal{O}(T^{-3/4})\right] \\
&+(C_{\psi}^{2}+\tfrac{2\sigma^{2}}{K})^{1/2}\left( \sqrt{t}\mathcal{O}(1) + \mathcal{O}(t^{3/2})\mathcal{O}(T^{-1})\right) 
\end{split}
\end{equation*}
Since $t\leq T$, the result follows.
\end{proof}
\section{Discussions on the parameters and constants}
\label{Sec:aaiissjhhdhdggdhdd}
This section aims to show the possibility of improving GOMSP by adapting the mirror map to the underlying feasible set. To this end, we compare the constants arising in the performance guarantees given in the previous section, both if the Euclidean norm -, and if the smoothed entropy serves as the regularizer. Throughout this section, we consider the constraint set: 
\begin{equation*}
\mathcal{X}=\lrbrace{x\in\real^{D}_{\geq 0}:~\sum_{i=1}^{D}x^{i}\leq B},\quad B\geq 1
\end{equation*}
%We consider the cases $\psi(\cdot)=\norm{\cdot}_{2}/2$ and the smoothed entropy:
%\begin{equation*}
%\psi_{\epsilon}^{\text{ent}}(x)=\sum_{i=1}^{D}(x_{i}+\epsilon)\ln(x_{i}+\epsilon).
%\end{equation*}
\paragraph{$\mathcal{D}(\X,\psi)$ and $\mathcal{D}_{\X}^{\psi}$}  To compute $\mathcal{D}(\X,\psi^{\ent}_{\epsilon})$, notice first that $\psi_{\epsilon}$ is strictly convex and therefore the minimizer of this function is an extreme point of $\X$. Consequently, we have for $\epsilon\leq 1$, $\max_{x\in\mathcal{X}}\psi^{\ent}_{\epsilon}(x)=B\ln(B)$. Now, by KKT-argumentations, it yields for $\epsilon\geq e^{-1}$, $\min_{x\in\mathcal{X}}\psi^{\ent}_{\epsilon}(x)=-D\epsilon\ln(\epsilon)$. Combining both observations, we have $\mathcal{D}(\X,\psi^{\ent}_{\epsilon})=B\ln(B)+D\epsilon\ln(\epsilon)$.
%\begin{equation*}
%\mathcal{D}(\X,\psi^{\ent}_{\epsilon})=B\ln(B)+D\epsilon\ln(\epsilon)
%\end{equation*}
In contrast, we have $\mathcal{D}(\X,\norm{\cdot}^{2}_{2}/2)=B^2$. So, using the smoothed entropy yields better dependency of $\mathcal{D}(\X,\psi)$ on $B$ ($B\ln(B)$ vs. $B^{2}$). However, $\mathcal{D}(\X,\psi)$ has a linear dependency on the $D$ which one fortunately can offset by choosing $\epsilon\in [e^{-1},1)$ large enough. The constant $\mathcal{D}_{\X}^{\psi}$ is irrelevant for our consideration, since it is equal $B$ for both choices of $\psi$.
\paragraph{$L_{\psi}$ and Sensitivity to Variation} Elementary computation yields $L_{\psi^{\ent}_{\epsilon}}= \max\lrbrace{\abs{1+\ln(\epsilon)},\abs{1+\ln(B+\epsilon)}}$. If $\epsilon\in [e^{-1},1]$, this quantity simplifies to $L_{\psi^{\ent}_{\epsilon}}= 1+\ln(B+\epsilon)$. In contrast, we have $L_{\norm{\cdot}_{2}}= B$. We see that choosing $\psi=\psi_{\epsilon}^{\ent}$ instead of $\psi=\norm{\cdot}_{2}^{2}/2$ might yields an improvement of the dependency of $L_{\psi}$ on $B$ ($\ln(B)$ vs. $B$) and therefore an improvement of the dependency of the GOMSP's regret performance on the variation. However, as a different choice of mirror map leads to a different norm measuring the variation, caution is required to this regard: With the choice $\psi=\psi_{\epsilon}^{\ent}$ we measure the variation by means of $\norm{\cdot}_{1}$, and with the Euclidean norm as regularizer, we measure the variation by means of $\norm{\cdot}_{2}$ which is in general smaller than $\norm{\cdot}_{1}$ (by at worst the factor $\sqrt{D}$). The discussion in this paragraph is irrelevant for the $\hCFit$ guarantee given in the previous section, since it is independent of the path variation.  
%\begin{equation*}
%L_{\psi^{\ent}_{\epsilon}}\leq \max\lrbrace{\abs{1+\ln(\epsilon)},\abs{1+\ln(B+\epsilon)}}.
%\end{equation*}
\paragraph{Constants related to loss function and penalty function}
Clearly, $C_{3,\psi}$ is equal in both choices of the regularizer. Since $\norm{\cdot}_{\infty}\leq\norm{\cdot}_{2}$,
%\begin{equation*}
%\norm{[\nabla (h\circ g)(x)]^{\T}\lambda}_{\infty}\leq \norm{[\nabla (h\circ g)(x)]^{\T}\lambda}_{2},
%\end{equation*}	
$C_{1,\psi_{\epsilon}^{\ent}}$ might be smaller than $C_{1,\norm{\cdot}_{2}^{2}/2}$. Similar argumentation yields that $C_{1,\psi^{\ent}_{\epsilon}}$ might be smaller $C_{1,\norm{\cdot}^{2}_{2}/2}$.
\paragraph{Strong Convexity and Noise}\label{Sec:aaiissjhhdhdggdhddNoise} It is immediate to see that $K_{\norm{\cdot}_{2}}=1$. Moreover, by Proposition \ref{Prop:aiaishshjfggfhdhddd}, we have $K_{\psi_{\epsilon}}=\tfrac{1}{B}$. So in case $B>1$, GOMSP with $\psi_{\epsilon}^{\text{ent}}$ as the regularizer might suffer more from noise amplification than GOMSP with the Euclidean norm as regularizer. Our advice concerning this issue is to normalize as far as possible the problem such that the restated problem has $B=1$. Regarding the power $\sigma_{\psi}$ of the persistent noise itself, we can leverage from choosing the smoothed entropy over the Euclidean norm as the regularizer of GOMSP. To see this, consider, for instance, an i.i.d. noise $(\xi_{t})_{t}$ where the coordinates of $\xi_{t}$ are independent standard Gaussian random variables. It holds that $\sigma_{\norm{\cdot}_{2}^{2}/2}^{2}$ is of order $D$. In contrast, $\sigma_{\psi^{\text{\ent}}_{\epsilon}}^{2}$ is of order $\ln(D)$ which is better.   

\section{Numerical Simulation}
In order to verify our theoretical findings, we test GOMSP and present in this section the result of our simulations. We first begin by stating the setting in our experiment.
\subsection{Online Problem Setting}
We test our method on a special case of the problem setting stated in Example \ref{Ex:Eco} with $20$ generators ($D=20$) and $10$ constraints ($R=10$) described in the following:

\paragraph{Feasible Set} We consider the feasible set $\X=\lrbrace{x\in\real^{D}:\sum_{i=1}^{D}x_{i}\leq B}$ with $B=1$. The reason for choosing $B=1$ is to prevent possible noise amplification by using a regularizer other than the Euclidean norm (see paragraph $d)$ in Section \ref{Sec:aaiissjhhdhdggdhdd}). For other setting where $B\neq 1$ one may reformulate the online problem such that the resulted feasible set has $B=1$.

\paragraph{Loss Function}
We consider the quadratic cost function $c_{t}^{(i)}(x^{(i)})=a_{t}^{(i)}(x^{(i)})^{2}+b_{t}^{(i)}x^{(i)}$,
where $a^{(i)}_{t}=0.5\sin(\pi t/50)+5+\tilde{a}_{t}$, with $(\tilde{a}_{t})$ is an i.i.d. random sequence uniformly distributed in the interval $[0,0.5]$, and where $b^{(i)}_{t}=0.5\sin(\pi t/100)+6+\tilde{b}_{t}$, where $(\tilde{b}_{t})$ is an i.i.d. random sequence uniformly distributed in the interval $[0,0.2]$. We set the demand service constant to be $\xi=20$. Our model for the time-varying non-stationary demand is given as $d_{t}=0.1\cos(\pi t/125)+0.7+\tilde{d}_{t}$, where $(\tilde{d}_{t})$ is an i.i.d. random sequence uniformly distributed in the interval $[0,0.2]$. 

\paragraph{Constraints}
The constraints are described by the quadratic functions $E_{t}^{i\rightarrow j}(x^{(i)})=c^{i\rightarrow j} (x^{(i)})^{2}+e^{i\rightarrow j}x^{(i)}$ where $c^{i\rightarrow j}$ and $e^{i\rightarrow j}$ are independent uniformly distributed random variable on the unit interval. We assume that the constraint thresholds are time-variant and non-stationary of the form $E^{\max,j}_{t}=0.05\cos(\pi t/50)+0.2+\tilde{e}_{t}$, where $(\tilde{d}_{t})$ is an i.i.d. random sequence uniformly distributed in the interval $[0,1]$.    
\subsection{Algorithm setting and Benchmarks}
All the method which we apply to the online learning problem receives a warm start of the amount of $40$ time-slots. Subsequently, we run the algorithms for $T=500$. We test GOMSP on the online learning problem describe previously with both the smoothed entropy with $\epsilon=0.5$ and Euclidean norm as a regularizer, where we set the step size to be $\gamma=0.1/\sqrt{T}$ and the regularization parameter to be $\alpha=15\gamma$. As choices of $h$ we consider $h=[\cdot]_{+}$ and $h=[\cdot]_{+}^{2}$.
\paragraph{Noisy Feedback}
To model the disturbance of the gradient feedback, we assume that learner can only observe the cost coefficients $(a_{t}^{(i)})_{i}$ and $(b_{t}^{(i)})_{i}$ at time $t$ up to a Gaussian random disturbance. Specifically, we assume at time $t$ that the learner sees $(\hat{a}_{t}^{i})_{i}$ and $(\hat{b}_{t}^{i})_{i}$, where $\hat{a}_{t}^{(i)}=a_{t}^{(i)}+\tilde{\xi}_{t+1}^{(i),1}$ and $\hat{b}_{t}^{(i)}=a_{t}^{(i)}+\tilde{\xi}_{t+1}^{(i),2}$, with $(\tilde{\xi}_{t}^{(i),1})_{i,t}$ (resp. $(\tilde{\xi}_{t}^{(i),2})_{i,t}$) is the sequence of i.i.d. mean zero Gaussian random variable with standard deviation $\sigma_{a}>0$ ($\sigma_{b}>0$). Throughout our simulation, we set $\sigma_{a}=0.2$ and $\sigma_{b}=1$.  

\paragraph{MOSP}
We compare GOMSP with the modified online saddle-point (MOSP) introduced in \cite{Chen2017} with fixed primal and dual step size equal to $\gamma\in\lrbrace{0.1,0.07,0.05}/\sqrt{T}$. In contrast to the works \cite{Chen2017,Chen2018}, we simulate MOSP with imperfect gradient feedback with the noise structure described in the previous paragraph.  
%we    The (noisy) update of this algorithm is given by:
%\begin{align*}
%X_{t+1}&=\Pi_{\mathcal{X}}\left[ X_{t}-\gamma\left( \hat{v}_{t}+\sum_{r=1}^{N} \left[ \nabla (h\circ g^{(r)}_{t})(X_{t})\right]\Lambda^{(r)}_{t}\right) \right] \\
%\Lambda_{t+1}&=\Pi_{\real^{R}_{\geq 0}}\left[(1-\alpha\gamma)\Lambda_{t}+\gamma h(g_{t}(X_{t}))\right]. 
%\end{align*}
\paragraph{ODG}
Furthermore, we also compare GOMSP with the stochastic dual gradient (SDG) method (see e.g., \cite{Tassiulas1992,Neely2010,Chen2017}), which we modify as follows: 
\begin{align*}
&X_{t+1}\in\argmin_{x\in\X}\hat{f}_{t}(x)+\inn{\Lambda_{t}}{g_{t}(x)}\\
&\Lambda_{t+1}=[\Lambda_{t}+\gamma g_{t}(X_{t})]_{+},
\end{align*} 
where $\hat{f}_{t}$ is the loss function with perturbed coefficients as described in Paragraph $a)$. This modification is for the sake of fairness in the comparison since the original SDG method requires non-causal knowledge and does not consider the possibilities of disturbance in the feedback. 
\subsection{Simulation Result}
\begin{figure}[h]
	\begin{center}
		\includegraphics[scale=0.6, trim={2.6cm 11cm 3.8cm 11.4cm}, clip]{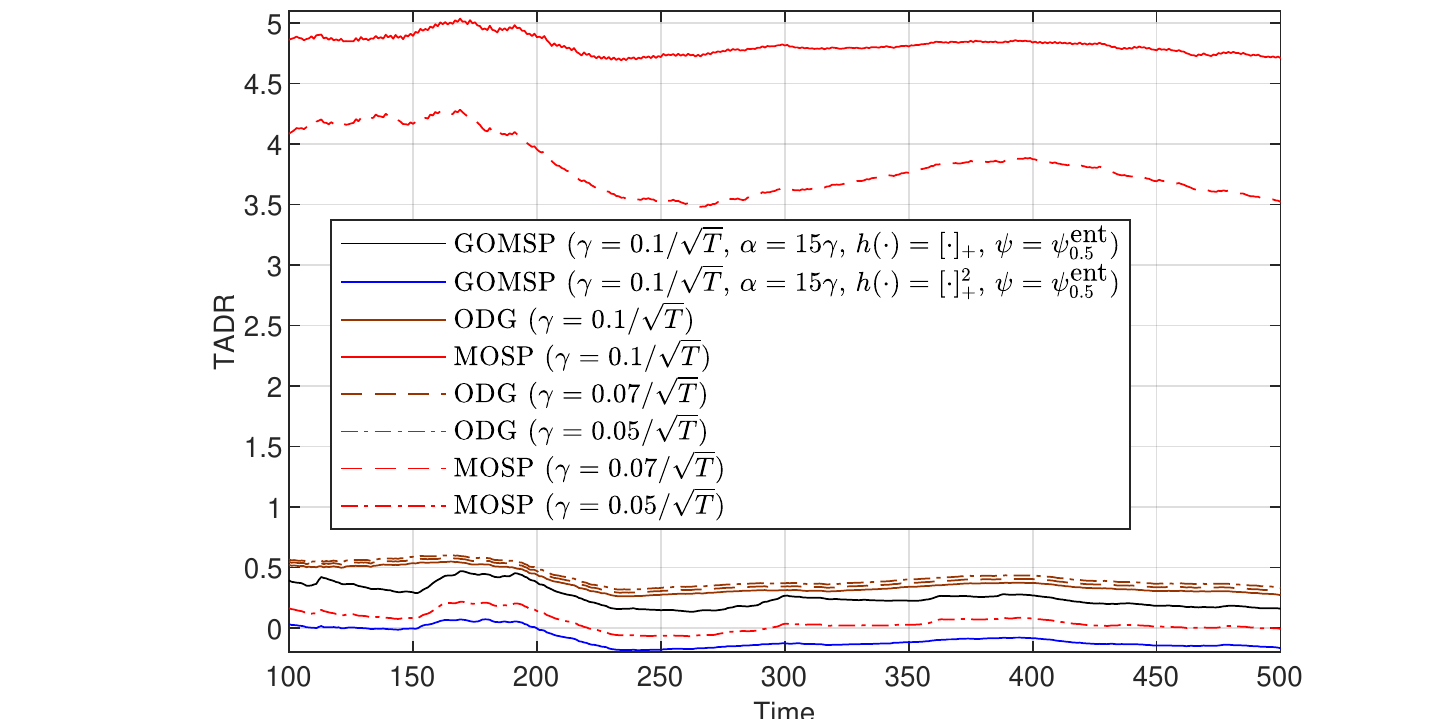}
	\end{center}
	\caption{Time-average Dynamic Regret (TADR) for GOSMP and benchmarks ODP and MOSP with perturbed cost $\sigma_{a}=0.2$ and $\sigma_{b}=1$.
		%, with either logit choice or (Euclidean) projection as mirror map and with/without the price regularization specified by MAARP.
	}
	\label{Fig:aoaojsjsjjddd2}
\end{figure}
\begin{figure}[h]
	\begin{center}
		\includegraphics[scale=0.65, trim={4cm 10cm 5cm 11.2cm}, clip]{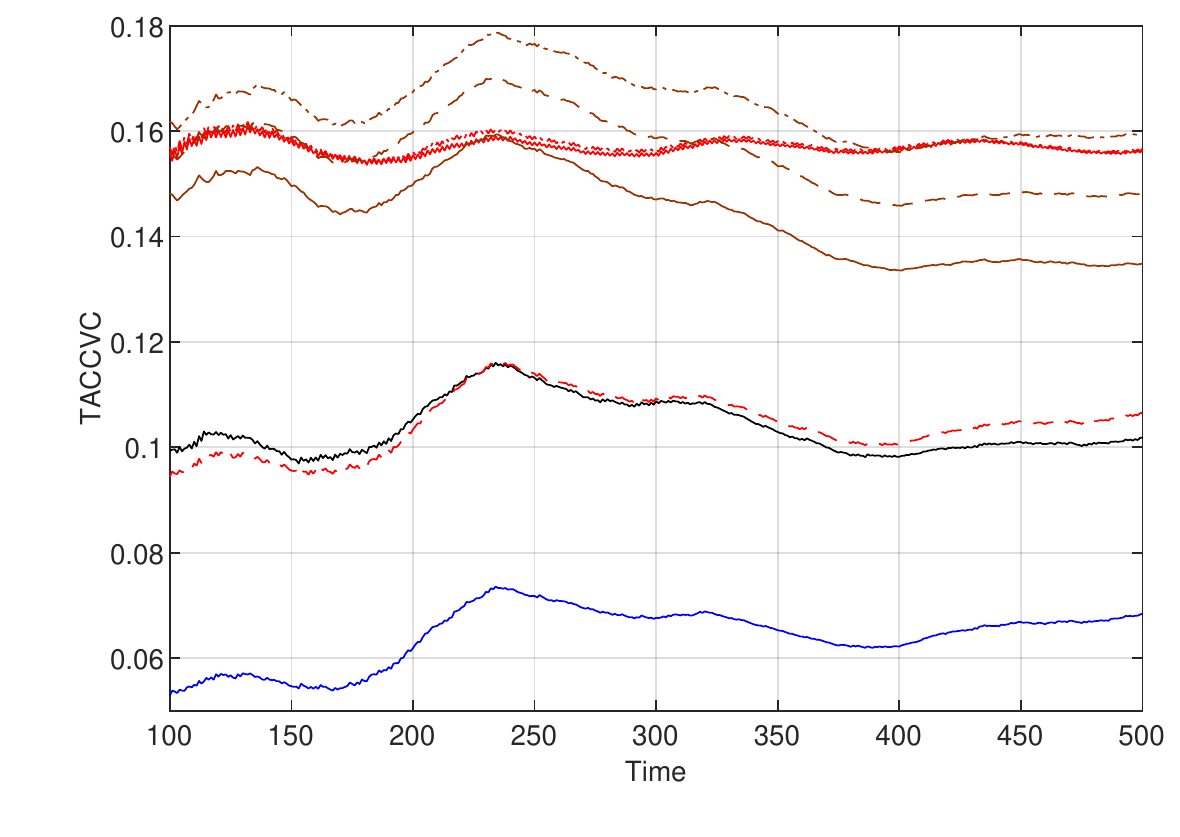}
	\end{center}
	\caption{Time-average clipped constraint violation (TACCV) for GOSMP and benchmarks ODP and MOSP with perturbed costs.
		%, with either logit choice or (Euclidean) projection as mirror map and with/without the price regularization specified by MAARP.
For legend see Fig. \ref{Fig:aoaojsjsjjddd2}.}
	\label{Fig:aoaojsjsjjddd1}
\end{figure}
\paragraph{Clipped Constraint Violation}
At first, we evaluate the Time average clipped constraint violation (TACCV) of the different methods given by:
\begin{equation*}
\tfrac{\sum_{t=1}^{t}\sum_{r=1}^{R}[g_{t}^{r}(X_{t})]_{+}}{t R}.
\end{equation*}
We see that in case the step sizes of the methods coincide ($\gamma=0.1$), GOMSP with smoothed entropy as the regularizer, independent of the choice of $h$, clearly outperform ODG and MOSP. However, we see that $h=[\cdot]_{+}^{2}$ yields the best performance. Moreover, even by reducing the step sizes of ODG and MOSP to $\gamma=0.07/\sqrt{T}$ and $\gamma=0.05/\sqrt{T}$ the corresponding TACCV is still higher than that of GOMSP.
\paragraph{Dynamic Regret}
Now we examine the dynamic regret of the methods averaged over time (TADR). We provide the plot of this quantity in Fig. \ref{Fig:aoaojsjsjjddd2}. Our method clearly outperform ODG w.r.t. to the performance measure TADR in the case where the step sizes of MOSP and its benchmarks coincide ($\gamma=0.1/\sqrt{T}$). However, running MOSP with smaller step size ($\gamma=0.05/\sqrt{T}$) it outperforms GOMSP with $h=[\cdot]_{t}$. This occurence can however be changed by choosing $h=[\cdot]_{+}^{2}$ since GOMSP possesses in this case the lowest and even negative dynamic regret.
\begin{figure}[h]
	\begin{center}
		\includegraphics[scale=0.7, trim={4cm 11cm 3.3cm 10cm}, clip]{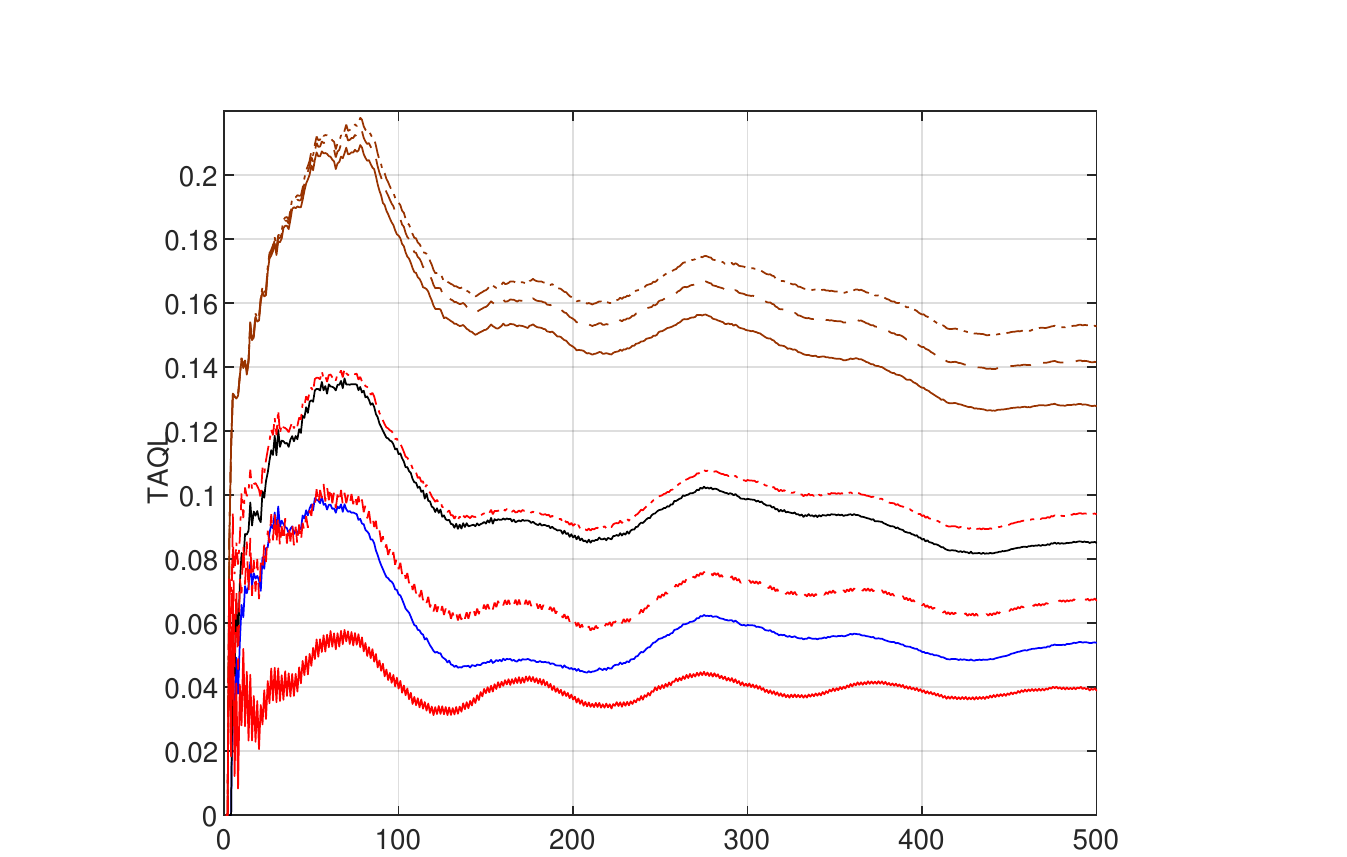}
	\end{center}
	\caption{Time-average Queue Length (TAQL) for GOSMP and benchmarks ODP and MOSP with perturbed cost $\sigma_{a}=0.2$ and $\sigma_{b}=1$. For legend see Fig. \ref{Fig:aoaojsjsjjddd2}.
		%, with either logit choice or (Euclidean) projection as mirror map and with/without the price regularization specified by MAARP.
	}
	\label{Fig:aoaojsjsjjddd3}
\end{figure}

\paragraph{Queue Length}
In our experiment, we also examine the queue length $(Q_{t})_{t}$ of GOMSP and its benchmarks, which is given by
\begin{equation*}
Q^{r}_{t+1}=[Q^{r}_{t}+g^{r}_{t}(X_{t})]_{+},
\end{equation*}
with $Q^{r}_{0}=0$. This quantity is relevant for applications where the current constraint violation can be compensated by previous actions that are strictly constraint fulfilling, which occurs in systems having the ability to buffer (see e.g. \cite{Chen2017}). Clearly, small TACCV does not imply small queue length since the former implies that the constraint violations remain small and the latter allows some substantial constraint violations of cost constraint values strictly smaller than the allowed threshold. We plot the time-average queue length (TAQL) $\sum_{r=1}^{R}Q_{t}^{r}/tR$ in Figure \ref{Fig:aoaojsjsjjddd3}. We see that MOSP with $\gamma=0.1/\sqrt{T}$ yields the lowest queue length. However, by observing its trajectory, this performance is caused by the fact that the update of MOSP highly and rapidly oscillates between states which strictly fulfilling the constraint and states violating the constraints. Such a behavior is not tolerable in technical applications since it might incur an additional switching cost (see e.g. \cite{Li2018}). Furthermore it is surprising in the face of the previous discussion on the difference between TACCVC and TAQL that ignoring the MOSP with $\gamma=0.1/\sqrt{T}$, it is possible that GOMSP may have the smallest TAQL.

\begin{figure}[h]
	\begin{center}
		\includegraphics[scale=0.6, trim={3.3cm 9cm 3.1cm 9cm}, clip]{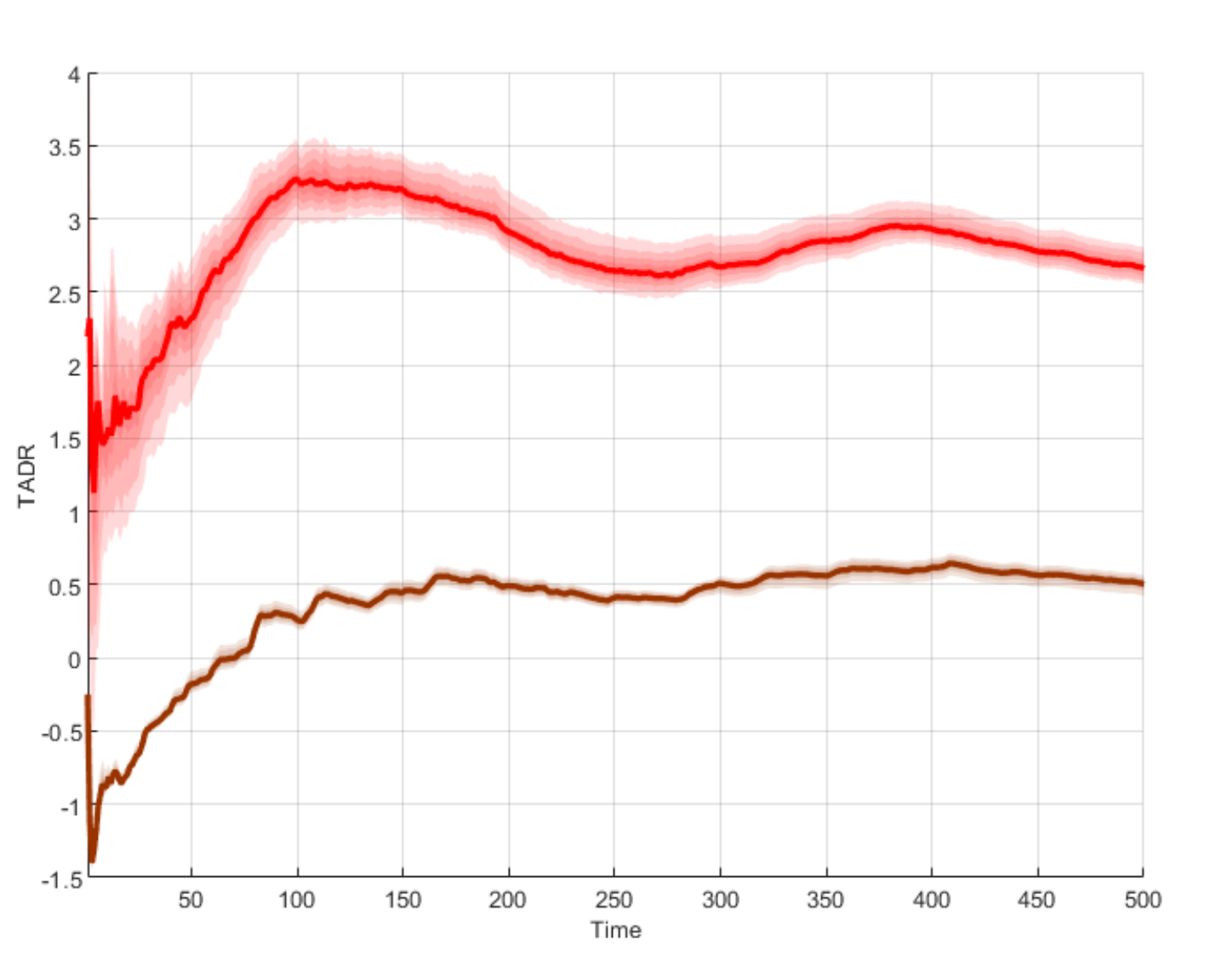}
	\end{center}
	\caption{TADR for GOSMP.
		Brown line corresponds to the sample average of TACCV in the smoothed entropy ($\epsilon=0.5$) case and red line resp. in the Euclidean case. Shaded areas are each corresponds to $25\%$-, $50\%$-, $75\%$-, and $90\%$-percentile.
	}
	\label{Fig:aoaojsjsjjddd4}
\end{figure}
\begin{figure}[h]
	\begin{center}
		\includegraphics[scale=0.9, trim={6cm 12cm 5cm 13cm}, clip]{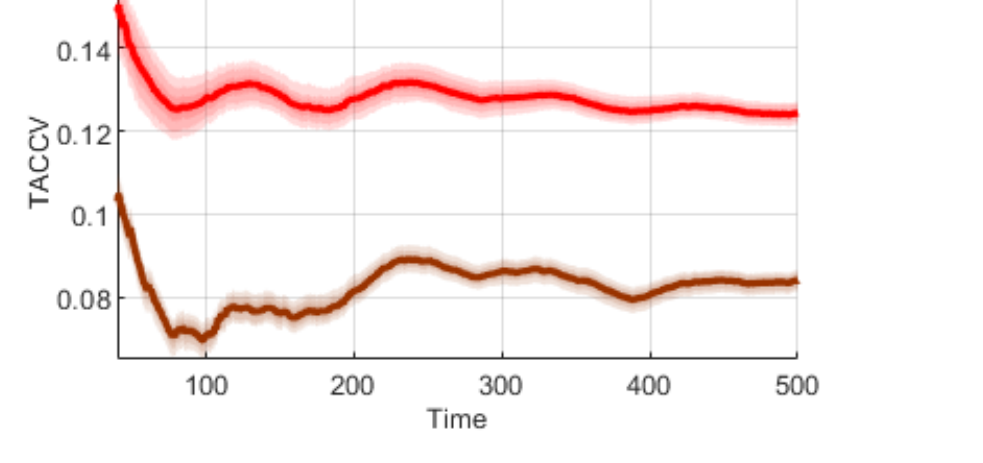}
	\end{center}
	\caption{TACCVC for GOSMP.
		Brown line corresponds to the sample average of TACCVC in the smoothed entropy ($\epsilon=0.5$) case and red line resp. in the Euclidean case. Shaded areas are each corresponds to $25\%$-, $50\%$-, $75\%$-, and $90\%$-percentile.
	}
	\label{Fig:aoaojsjsjjddd5}
\end{figure}
\begin{figure}[h]
	\begin{center}
		\includegraphics[scale=0.7, trim={6cm 12cm 5cm 12cm}, clip]{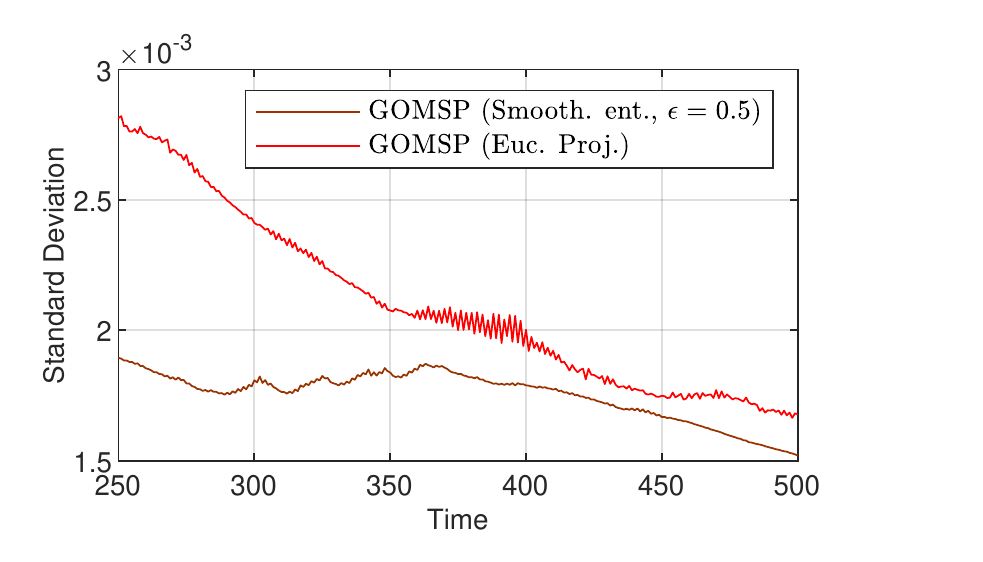}
	\end{center}
	\caption{Standard Deviation of GOSMP
		%, with either logit choice or (Euclidean) projection as mirror map and with/without the price regularization specified by MAARP.
	}
	\label{Fig:aoaojsjsjjddd6}
\end{figure}
\paragraph{Impact of Mirror Map Choice}
At last, we are interested in investigating to what extent does the choice of the mirror map impacts the performance of GOSMP. Toward this end, we perform GOSMP with Euclidean projection and the smoothed entropy with $\epsilon=0.5$ as regularizers. In both cases, we choose $\gamma=0.1/\sqrt{T}$, $\alpha=0.1/\sqrt{T}$, and $h=[\cdot]_{+}$. We simulate both instances of GOSMP with $200$ gradient noise samples. Figure \ref{Fig:aoaojsjsjjddd4} depicts the dynamic regret of our simulation. There the thick line corresponds to the sample average of the trajectories, and the shaded line specifies the area where $25\%$, $50\%$, $75\%$, and $90\%$ of the samples are. A clear trend which we can observe is that the TADR of GOSMP with smoothed entropy as regularizer is significantly lower than the TADR of GOSMP with Euclidean projection as regularizer. We believe that this effect aligns with the discussion made in paragraph $c)$ in Section V. Moreover we observe that the TADR of GOSMP with smoothed entropy as regularizer is more volatile than that of GOSMP with Euclidean projection as regularizer. This observations confirms the hypothesis that using a mirror map other than the Euclidean one results in more robust algorithm behavior. We also observe similar trends in the resource-aware behavior of GOSMP (see Figure 	\ref{Fig:aoaojsjsjjddd5} and \ref{Fig:aoaojsjsjjddd6}). However, the effect of noise reduction is less pronounced comparing to that of TADR. 
%\section{Conclusion and Future Work}
%\textcolor{red}{
%\begin{itemize}
%	\item We introduce $\hCFit$.
%	\item Improvement by using the general notion of Mirror descent.
%	\item Long term constraint without $h$ might violate ramping constraint.
%\end{itemize}
%}
\appendix
\subsection{Missing Proofs in Section \ref{Sec:Perf}}
\label{Subsec: AppProof}
The proof of Lemma 	\ref{Lem:aaiisshdggdhgdhsssss} is straightforward following \cite{Zinkevich2003}:
\begin{proof}[Proof of Lemma 	\ref{Lem:aaiisshdggdhgdhsssss}]
	By inserting the primal iterate of the GOMSP into the bound given in Proposition \ref{Prop:aaisshhfjffjfjfff}, by using triangle inequality, by the inequality $(\sum_{i=1}^{K}a_{i})^{2}\leq K\sum_{i=1}^{K}a_{i}^{2}$, it is straightforward to obtain:
	\begin{align*}
	&\Delta\mathcal{E}^{1}_{t}(x)
	\leq-\gamma\inn{X_{t}-x}{\nabla f_{t}(X_{t})+[\nabla (h\circ g_{t})(X_{t})]^{T}\Lambda_{t}}\\&+\gamma\tilde{\xi}_{t+1}+\frac{1}{K}\left(\gamma^{2}C_{1}^{2}\norm{\Lambda_{k}}_{2}^{2}+2\gamma^{2}(C_{2}^{2}+\norm{\xi_{t+1}}_{*}^{2})\right)
	\end{align*}
\end{proof}

Our aim now is to proof Lemma \ref{Lem:aaiisshdggdhgdhsssss2}. It is an immediate consequence of the following auxiliary statements: 
\begin{lemma}
	\label{Lem:aiashshsgdgdssdd}
	It holds:
	\begin{align*}
	%\tfrac{\norm{\Lambda_{t+1}}_{2}^{2}}{2}-\tfrac{\norm{\Lambda_{t}}_{2}^{2}}{2}
	\Delta\mathcal{E}^{2}_{t}
	\leq \gamma\inn{\Lambda_{t}}{h(g_{t}(X_{t}))}-(\alpha\gamma-\alpha^{2}\gamma^{2})\norm{\Lambda_{t}}^{2}_{2}+\gamma^2 C_{3}^{2},
	\end{align*}
	where $C_{3}>0$ is a constant satisfying \eqref{Eq:asjsjskdhdggddfsssddd}.
\end{lemma}
\begin{proof}
	It holds:
	\begin{equation*}
	\begin{split}
	&\norm{\Lambda_{\tau+1}}^{2}_{2}=\norm{\Pi_{\real^{R}_{\geq 0}}\left[ (1-\alpha\gamma)\Lambda_{\tau}+\gamma h(g_{\tau}(X_{\tau}))\right] }^{2}_{2}\\
	&\leq\norm{\Lambda_{\tau}+\gamma h(g_{\tau}(X_{\tau}))-\alpha\gamma\Lambda_{\tau} }^{2}_{2}=\norm{\Lambda_{\tau}}^{2}\\
	&+2\left[\gamma  \inn{\Lambda_{\tau}}{h(g_{\tau}(X_{\tau}))}-\alpha\gamma\norm{\Lambda_{\tau}}^{2}\right] +\gamma^{2}\norm{h(g_{\tau}(X_{\tau}))-\alpha\Lambda_{\tau}}_{2}^{2},
	\end{split}
	\label{Eq:aaaisjjjshshhhdddffff}
	\end{equation*}
	where the inequality follows from the usual property of the Euclidean projection operator.
	Triangle inequality, the inequality $(a+b)^{2}\leq 2a+2b$, and \eqref{Eq:asjsjskdhdggddfsssddd} give $\norm{h(g_{\tau}(X_{\tau}))-\alpha\Lambda_{\tau}}_{2}^{2}\leq 2\left( C_{3}^{2}+\alpha^{2}\norm{\Lambda_{\tau}}_{2}^{2}\right)$. 
	%\begin{equation*}
	%\begin{split}
	%\norm{h(g_{\tau}(X_{\tau}))-\alpha\Lambda_{\tau}}_{2}^{2}\leq 2\left( C_{3}^{2}+\alpha^{2}\norm{\Lambda_{\tau}}_{2}^{2}\right) 
	%\end{split}
	%\end{equation*}
	So combining all the derived inequalities, we obtain Lemma Proof of Lemma \ref{Lem:aaiisshdggdhgdhsssss2}.
\end{proof}
\begin{lemma}
	\label{Lem:aiashshsgdgdssdd2}
	Suppose that $h$ is monotone and $g$ is convex. Let be $\tau$ fixed. It holds for any $\tilde{x}\in\mathcal{Q}_{\tau}$:
	\begin{align*}
	&\inn{\Lambda_{\tau}}{h(g_{\tau}(X_{\tau}))}\leq \inn{\left[ \nabla (h\circ g_{\tau})(X_{\tau})\right]^{\T} \Lambda_{\tau}}{X_{\tau}-\tilde{x} }.
	\end{align*}
\end{lemma}
\begin{proof}
	Let be $x\in\mathcal{X}$. Since $h$ is monotone and $g^{(r)}$ is convex for all $r\in [R]$, it follows that $h\circ g^{(r)}$ is convex. This and the fact that $\Lambda_{\tau}\geq 0$ for all $\tau$ gives:
	\begin{equation*}
	\begin{split}
	&\inn{\Lambda_{\tau}}{h(g(X_{\tau}))}\\
	&\leq\inn{\Lambda_{\tau}}{h(g_{\tau}(x))} -\inn{\Lambda_{\tau}}{\partial (h\circ g_{\tau})(X_{\tau})(x-X_{\tau}) }\\
	&=\inn{\Lambda_{\tau}}{h(g_{\tau}(x))} +\inn{\left[ \partial (h\circ g_{\tau})(X_{\tau})\right]^{\T} \Lambda_{\tau}}{X_{\tau}-x }
	\end{split}
	\end{equation*}	
	Since $h$ is monotone, we have for $\tilde{x}\in\mathcal{Q}_{\tau}\subset\mathcal{X}$, $h(g_{\tau}(\tilde{x}))\leq 0$. Consequently since $\Lambda_{\tau}\geq 0$, it yields $\inn{\Lambda_{\tau}}{h(g_{\tau}(\tilde{x}))}\leq 0$. Combining all the computations, we obtain the desired result.
\end{proof}
Consequently by combining Lemmas \ref{Lem:aiashshsgdgdssdd} and \ref{Lem:aiashshsgdgdssdd2} we obtain Lemma \ref{Lem:aaiisshdggdhgdhsssss2}.
\subsection{Properties of Mirror Map and Fenchel coupling}
The following Proposition which is a folklore in convex analysis gives some basic properties of the mirror map:
\begin{proposition}
	\label{Prop:aiaishshjfggfhdhddd}
	Let $\psi$ be a $K$-strongly convex regularizer on a compact convex subset $\mathcal{Z}$ of a Euclidean normed space $\mathcal{V}$ inducing the mirror map $\Phi:\V^{*}\rightarrow\Z$, and let $\psi^{*}:\V^{*}\rightarrow\real$, $y\mapsto\max_{x\in\Z}\left\{\left\langle x,y\right\rangle-\psi(x)\right\}$ be the convex conjugate of $\psi$. Then:
	\begin{enumerate}
		\item $x=\Phi(y)$ if and only if $y\in\partial \psi(x)$. In particular $\text{im}(\Phi)=\text{dom}(\partial \psi)\supseteq \text{relint}(\Z)$.
		\item $\psi^{*}$ is differentiable on $\V^{*}$ and $\nabla \psi^{*}(y)=\Phi(y)$.
		\item $\Phi$ is $(1/K)$-Lipschitz continuous.
		\item $\psi$ is $1/\norm{\mathcal{Z}}_{*}$-strongly convex w.r.t. $\norm{\cdot}$.
	\end{enumerate}
\end{proposition}
\begin{proof}
For a proof of 1)-3), see e.g. Theorem 23.5 in \cite{Rockafellar1970} and Theorem 12.60(b) in \cite{Rockafellar1998}.

For the statement 4), notice that $\norm{\nabla\psi^{*}(y)}_{*}=\norm{\Phi(y)}_{*}\leq \norm{\mathcal{Z}}_{*}$ where the inequality follows from the fact that $\Phi$ is a mapping to $\mathcal{Z}$. Therefore $\psi^{*}$ is $\norm{\Z}_{*}$-strongly smooth and Strong/smooth duality Theorem (see e.g. Theorem 3 in \cite{Kakade2012}) asserts the desired statement.
\end{proof} 
Some useful properties of the Fenchel coupling is stated in the following (for proof see \cite{Mertikopoulos2016}):
\begin{proposition}
	\label{Prop:aaisshhfjffjfjfff}
	Let $F$ be the Fenchel coupling induced by a $K$-strongly convex regularizer on a compact convex subset $\mathcal{Z}$ of a Euclidean normed space $\V$. For $p\in\mathcal{
		Z}$, $y,y^{'}\in \V^{*}$, we have:
	\begin{enumerate}
		%\item $F(p,y)=D(p,\Phi(y))$ if $\Phi(y)\in \text{int}{(\X)}$.
		\item $F(p,y)\geq (K/2)\norm{\Phi(y)-p}^{2}$
		\item $F(p,y^{'})\leq F(p,y)+\inn{\Phi(y)-p}{y^{'}-y}+(1/2K)\norm{y^{'}-y}^{2}_{*}$
	\end{enumerate}
\end{proposition}
\bibliographystyle{IEEEtran}

\end{document}